\documentclass{article}

\usepackage{amsfonts}
\usepackage{amssymb}
\usepackage{amsmath}
\usepackage{amsthm}
\usepackage{amsopn}
\usepackage{amsbsy}

\usepackage{graphicx}
\graphicspath{{./pictures/}}

\usepackage{subfigure}
\usepackage{bm}
\usepackage{mathrsfs}
\usepackage{fancyhdr}

\usepackage[a4paper]{geometry}

\newtheorem{theorem}{Theorem}

\newtheorem{lemma}[theorem]{Lemma}
\newtheorem{corollary}[theorem]{Corollary}
\theoremstyle{definition}
\newtheorem{remark}[theorem]{Remark}

\numberwithin{theorem}{section}

\newcommand{\TheTitle}{Divergence-conforming discontinuous Galerkin finite elements for Stokes eigenvalue problems} 
\newcommand{\TheShortTitle}{$H^{\textrm{div}}$-DGFEM for Stokes eigenvalue problems} 
\newcommand{\TheAuthors}{Joscha Gedicke and Arbaz Khan}

\date{}

\usepackage[colorlinks=true,%
citecolor=black,%
filecolor=black,%
linkcolor=black,%
urlcolor=blue,%
pdfauthor={\TheAuthors},%
pdftitle={\TheTitle}]{hyperref}

\newcommand{\vertiii}[1]{{\left\vert\kern-0.25ex\left\vert\kern-0.25ex\left\vert #1 
    \right\vert\kern-0.25ex\right\vert\kern-0.25ex\right\vert}}

\title{{\TheTitle}%
\thanks{
The first author has been funded by the Austrian Science Fund (FWF) 
through the project P 29197-N32.
The second author has been funded
by the Mathematics Center Heidelberg (Match) at University of Heidelberg, Germany and EPSRC grant EP/P013317.}}

\author{
Joscha Gedicke\thanks{Faculty of Mathematics, University of Vienna, Oskar-Morgenstern-Platz 1, 1090 Vienna, Austria
(\mbox{joscha.gedicke@univie.ac.at}).}
\and
Arbaz Khan\thanks{Corresponding author: 
School of Mathematics,  The University of Manchester, 
M13 9PL Manchester, UK (\mbox{arbaz.khan@manchester.ac.uk}).}
}

\begin{document}

\maketitle

\begin{abstract}
In this paper, we present a divergence-conforming discontinuous Galerkin finite element method 
for Stokes eigenvalue problems.
We prove a priori error estimates for the eigenvalue and eigenfunction errors and
present a robust residual based a posteriori error estimator. The a posteriori error estimator is proven to be reliable and
(locally) efficient in a mesh-dependent velocity-pressure norm. 
We finally present some numerical examples that verify the a priori convergence rates
and the reliability and efficiency of the residual based a posteriori error estimator.
\end{abstract}

{\small\noindent\textbf{Keywords}
a priori analysis, a posteriori analysis, H-div conforming DG finite element, Stokes problem, eigenvalue problem

}

\section{Introduction}\label{intro}
In fluid mechanics, eigenvalue problems are of great importance because of their role for the stability analysis of fluid flow problems.
Hence, the development of numerical methods for the Stokes problem, as a model for incompressible fluid flow, is of great interest.
For example in \cite{PHYHXF}, 
several stabilized finite element methods for the Stokes eigenvalue problem are considered by Huang et al. 
A finite element analysis of a pseudo stress formulation for the Stokes eigenvalue problem is proposed by  Meddahi et al. \cite{MSDMRR}. 
\par
Currently, there are only very few results on the a posteriori error analysis for the Stokes eigenvalue problem available in the literature. 
An a posteriori error analysis based on residual a posteriori error estimators  for the finite element discretization of
the Stokes eigenvalue problem is proposed by Lovadina et al. \cite{LCMLRS}. 
Some superconvergence results and the related recovery type a posteriori error estimators for the Stokes eigenvalue problem is presented by Liu et al. \cite{LHGWWSYN}  based on a projection method. 
In \cite{AMGVM}, Armentano et al.  introduced a posteriori error estimators for stabilized low-order mixed finite elements and in \cite{HJZZYY}, Han et al. presented a residual type a posterior error estimator for a new adaptive mixed finite element method for the Stokes eigenvalue problem. 
In \cite{PH}, Huang presents a posteriori lower and upper eigenvalue bounds for the Stokes eigenvalue problem for two stabilized finite element methods based on the lowest equal-order finite element pair. Recently, we have developed an a posteriori error analysis for the Arnold-Winther mixed finite element method of the Stokes eigenvalue problem in \cite{JA} using the stress-velocity formulation. 
\par
Cockburn et al. \cite{CBKGSD,BCGKDS} derived the divergence-conforming discontinuous Galerkin finite element method. In \cite{HSW}, Houston et al. presented an a posteriori error estimation for mixed discontinuous Galerkin approximations of the Stokes problem. 
Kanschat et al. \cite{GKDS} proposed a posteriori error estimates for divergence-€free discontinuous Galerkin approximations of the Navier-€"Stokes equations.
 Multigrid methods for $H^{\textrm{div}}$-conforming discontinuous Galerkin ($H^{\textrm{div}}$-DG) finite element methods for the Stokes equations are proposed by Kanschat et al. \cite{KGYM}. Recently, Kanschat et al. \cite{GKNS} presented the relation between the $H^{\textrm{div}}$-DG finite element method for the Stokes equation and the $C^0$ interior penalty finite element method for the biharmonic problem. 
\par
In this paper, we introduce an $H^{\textrm{div}}$-DG finite element method  for Stokes eigenvalue problems. 
We derive a priori error estimates for the eigenvalue and eigenfunction errors.
We present a robust a posteriori error analysis of the $H^{\textrm{div}}$-DG finite element method 
and derive upper and local lower bounds for the velocity-pressure error which is measured in terms of the mesh-dependent DG  norm.  The proposed a posteriori error estimator is robust in the sense that the ratio of upper and lower bounds is independent of the viscosity coefficient and the local mesh size.
\par
For simplicity of the presentation we restrict the analysis to the case of a simple eigenvalue
$\lambda$. The results can be applied to multiple eigenvalues by extending the
given analysis to subspaces of eigenvectors that belong to the same multiple eigenvalue.
The a posteriori error estimator can be extended to multiple eigenvalues
in that the squared sum over all estimators of discrete eigenfunctions approximating the same multiple eigenvalue provides
an upper bound of the eigenvalue error up to higher order terms.
\par
The paper is organised as follows: the necessary notation and the $H^{\textrm{div}}$-DG formulation of the Stokes eigenvalue problem is presented in Section \ref{formu}. 
In Section \ref{sec:2}, the a priori error analysis is discussed.  The a posteriori error analysis  is developed in Section \ref{aposterrestana}. 
Finally, Section \ref{comre} is devoted to present some numerical results for uniform and adaptive mesh refinement.
\par

\section{Preliminaries} \label{formu}
\subsection{Notation}
Define $\bm{v}=(v_1,v_2)^t\in\mathbb{R}^2$ and $\bm{\tau}=(\tau_{ij})_{2\times 2}$,  then
\begin{align*}
\nabla\bm{v}&:=\Bigg(\begin{array}{cc}
\frac{\partial v_1}{\partial x}&\frac{\partial v_1}{\partial y}\\
\frac{\partial v_2}{\partial x}&\frac{\partial v_2}{\partial y}
\end{array}\Bigg),\qquad\text{and}\qquad {\textrm{div}}\bm{\tau}:=\Bigg(\begin{array}{c}
\frac{\partial \tau_{11}}{\partial x}+\frac{\partial \tau_{12}}{\partial y}\\
\frac{\partial \tau_{21}}{\partial x}+\frac{\partial \tau_{22}}{\partial y}
\end{array}\Bigg).
\end{align*}
Let $H^s(\omega)$ be the standard Sobolev space  with the associated norm $\|\cdot\|_{s,\omega}$ for $s\ge0$. In case of $\omega=\Omega$, we use $\|\cdot\|_{s}$ instead of $\|\cdot\|_{s,\Omega}$. Let $H^{-s}(\omega) :=(H^{s}(\omega))^*$ be the dual space of $H^{s}(\omega)$. Now we extend the definitions to vector and matrix-valued functions. Let $\bm{H}^{s}(\omega)=\bm{H}^{s}(\omega;\mathbb{R}^2)$ and  $\bm{H}^{s}(\omega,\mathbb{R}^{2\times 2})$ be the Sobolev spaces over the set of $2$-dimensional vector and $2\times 2$ matrix-valued function, respectively.
The symbols $\lesssim$ and $\gtrsim$ are used to denote bounds which are valid up to positive constants independent of the local mesh size.
\par
Throughout the paper, we consider the following spaces $L^2_0(\Omega)$, $\bm{H}^{\textrm{div}}(\Omega)$, $\bm{H}^{\textrm{div}}_0(\Omega)$ 
and $\bm{H}^{\textrm{div}}(\Omega,\mathbb{R}^{2\times 2})$ which are defined as follows:
\begin{align*}
L^2_0(\Omega)&:=\{v\in {L}^2(\Omega) \;|\; \int_{\Omega}v \,dx=0\},\\
\bm{H}^{\textrm{div}}(\Omega)&:=\{\bm{v}\in \bm{L}^2(\Omega) \;|\; \nabla\cdot \bm{v}\in {L}^2(\Omega)\},\\
\bm{H}^{\textrm{div}}_0(\Omega)&:=\{\bm{v}\in \bm{H}^{\textrm{div}}(\Omega) \;|\; \bm{v}\cdot\bm{n}=0 \;\;\mbox{on}\;\;\partial\Omega\},\\
\bm{H}^{\textrm{div}}(\Omega,\mathbb{R}^{2\times 2})&:=\{\bm{\tau}\in \bm{L}^2(\Omega;\mathbb{R}^{2\times2}) \;|\; {\textrm{div}}\, \bm{\tau}\in \bm{L}^2(\Omega)\}.
\end{align*} 

\subsection{Weak formulation of the Stokes eigenvalue problem}
Let $\bm{u}$ be the velocity, $p$ the pressure, $\nu>0$ the (constant) viscosity, and $\Omega\subset \mathbb{R}^2$ be a bounded, and connected Lipschitz domain. 
Consider the velocity-pressure formulation of the Stokes eigenvalue problem:
find an eigentripel $(\bm{u},p,\lambda)$, $\bm{u}\neq\bm{0}$, such that
 \begin{align}\label{pre12}
 \begin{split}
 -\nu \bigtriangleup \bm{u}+\nabla p &=\lambda \bm{u}\quad \mbox{in}\;\Omega,  \\
\nabla\cdot \bm{u}&=0\quad \;\;\;\mbox{in}\;\Omega,\\
 \bm{u}&=\bm{0}\quad \;\;\;\mbox{on}\;\partial\Omega,
 \end{split}
\end{align}  
with the compatibility relation
\begin{align*}
\int_{\Omega} p\;dx=0.
\end{align*}
The weak formulation of the Stokes eigenvalue problem \eqref{pre12} reads: 
find $(\bm{u},p,\lambda) \in \bm{H}^1_0(\Omega)\times L^2_0(\Omega) \times \mathbb{R}_+$ such that $\|\bm{u}\|_0=1$ and
\begin{align}\label{variational}
\begin{split}
\nu(\nabla \bm{u},\nabla \bm{v}) - (p,\nabla\cdot\bm{v}) &= \lambda(\bm{u},\bm{v})\quad\forall\bm{v}\in \bm{H}^1_0(\Omega),\\
(q,\nabla\cdot\bm{u})  &= 0\qquad\quad\;\;\forall q\in L^2_0(\Omega).
\end{split}
\end{align}
We can formulate the weak formulation of \eqref{variational}  in a global form as: find $(\bm{u},p,\lambda)\in \bm{H}_0^1(\Omega)\times L^2_0(\Omega)\times\mathbb{R}_+$ 
such that $\|\bm{u}\|_0=1$ and
\begin{align}\label{stpre11}
\mathcal{A}(\bm{u},p;\bm{v},q)=\lambda(\bm{u},\bm{v})\quad\forall (\bm{v},q)\in \bm{H}_0^1(\Omega)\times L^2_0(\Omega), 
\end{align}
where 
\begin{align*}
\mathcal{A}(\bm{u},p;\bm{v},q)=\nu(\nabla\bm{u},\nabla\bm{v})-(p,\nabla\cdot\bm{v})-(q, \nabla\cdot\bm{u}).
\end{align*}

\subsection{Meshes, trace operators and discrete spaces}\label{mifem}
We suppose that the domain $\Omega$ is decomposed by a subdivision $\mathcal{T}_h$ into a mesh of shape-regular rectangular cells $K$. 
Let $\mathcal{E}_h$ denote the set of edges, $\mathcal{E}^i_h$ the set of interior edges, and $\mathcal{E}^\partial_h$ the set of boundary edges of $\mathcal{T}_h$.  
We restrict ourself to one-irregular meshes  $\mathcal{T}_h$ in which each interior edge $E\in\mathcal{E}^i_h$ 
may contain at most one hanging node in the midpoint of $E$.
\par
For a given mesh $\mathcal{T}_h$, the notions of broken spaces for the continuous and differentiable function spaces are denoted as $C^0(\mathcal{T}_h)$ and $H^s(\mathcal{T}_h)$ which are the spaces such that the restriction to each mesh cell $K\in\mathcal{T}_h$ is in $C^0(K)$ and $H^s(K)$, respectively. 
\par
Let $K_\pm\in\mathcal{T}_h$ be two mesh cells which share a common edge $E=K_+\cap K_-\in\mathcal{E}^i_h$. 
The traces of functions $v\in C^0(\mathcal{T}_h)$ on $E$ from $K_\pm$ are defined as $v_\pm$, respectively. Then the sum operator is defined as
\begin{align*}
[\![v]\!]=v_+ + v_-.
\end{align*}
Let $\bm{n}_\pm$ be the unit outward normal vector to $K_\pm$, respectively. Then the sum operator turns into the jump operator, such that for $\bm{v}\in C^1(\mathcal{T}_h;\mathbb{R}^2)$
\begin{align*}
[\![\partial\bm{v}/\partial\bm{n}]\!]=\nabla(\bm{v}_+ - \bm{v}_-)\bm{n}_+, 
\quad\text{and}\quad[\![\bm{v}\otimes\bm{n}]\!]=(\bm{v}_+ - \bm{v}_-)\otimes\bm{n}_+.
\end{align*}
For boundary edges $E=K_+\cap\partial\Omega$ we set $[\![\bm{v}]\!]=\bm{v}_+$ and
with $\nabla_h$ we denote the local application of the gradient
$(\nabla_h\bm{v})|_K = \nabla(\bm{v}|_K)$ on each $K\in\mathcal{T}_h$.
\par
We define ${Q}_k(K),{Q}_k(K)^d$ and ${Q}_k(K)^{d\times d}$ as the space of scalar, vector and tensor valued polynomials on $K$ of partial degree at most integer $k\geq 1$.
\par
Choose $\bm{V}_h$ as a discrete subspace of $\bm{H}^{\textrm{div}}_0(\Omega)$ as  
\begin{align}\label{subsapce1}
\bm{V}_h=\{\bm{v}\in \bm{H}^{\textrm{div}}_0(\Omega)\;|\;\forall K\in \mathcal{T}_h : \bm{v}|_K \in RT_k(K) \mbox{ for } k\ge 1\},
\end{align}
where $RT_k(K) :=  \mathcal{P}_{k+1,k}(K)\times \mathcal{P}_{k,k+1}(K)$ is  the Raviart-Thomas space of degree $k\geq 1$, where
$\mathcal{P}_{r,s}(K)$  denotes the space of 
the polynomial functions on $K$ of degree at most $r>0$ in $x_1$ and at most $s>0$ in $x_2$.
Moreover, let $Q_h$ be the discrete space of $L^2_0(\Omega)$ such that
\begin{align}\label{subsapce2}
Q_h=\{v\in L^{2}_0(\Omega)\;|\;\forall K\in \mathcal{T}_h : v|_K \in {Q}_k(K) \mbox{ for } k\ge 1\}.
\end{align}
An important property of the pair $\bm{V}_h\times Q_h$ is as follows: on the meshes considered, 
\begin{align*}
\nabla\cdot\bm{V}_h\subset Q_h,
\end{align*}
see \cite{CBKGSD} for more details.
As a consequence we have that the discrete velocity field $\bm{u}_h$ is exactly divergence free.
\begin{remark}
The inf-sup stability of discretizations with hanging nodes using Raviart-Thomas finite elements is in part still an
open question. In \cite{DSCSAT}, there exists a stability proof only for the
pair $RT_k \times Q_k$ defined in \eqref{subsapce1} and \eqref{subsapce2} with $k\ge2$  for quadrilaterals with one-irregular meshes. However, we conjecture from our
computational results that stability also holds for $k = 1$. Moreover, the stability result for the divergence-free elements proposed in \cite{BCGKDS} is not available for triangles with hanging nodes. On the other hand, locally
refined triangular meshes without hanging nodes can be obtained using bisection. The results below
are all to be read in view of the restrictions cited in this remark.
\end{remark}
\begin{remark}
The analysis of this paper also applies directly to divergence-free $BDM_k/P_{k-1}$ finite elements on regular triangular meshes.
\end{remark}

\subsection{$H^{\textrm{div}}$-DG formulation for the Stokes eigenvalue problem}
The discrete weak formulation of problem \eqref{pre12} reads: find $(\bm{u}_h,p_h,\lambda_h)\in \bm{V}_h\times Q_h\times\mathbb{R}_+$ such that 
$\|\bm{u}_h\|_0=1$ and
\begin{align}\label{hdiv11}
\mathcal{A}_h(\bm{u}_h,p_h;\bm{v}_h,q_h)=\lambda_h(\bm{u}_h,\bm{v}_h) \quad\forall (\bm{v}_h,q_h)\in \bm{V}_h\times Q_h,
\end{align}
where 
\begin{align*}
\mathcal{A}_h(\bm{u}_h,p_h;\bm{v}_h,q_h)=a_h(\bm{u}_h,\bm{v}_h)-(p_h,\nabla\cdot\bm{v}_h)-(q_h, \nabla\cdot\bm{u}_h).
\end{align*}
Here, $a_h(\cdot,\cdot)$ is the bilinear form
defined as
\begin{align}
a_{h}(\bm{u},\bm{v})&=\nu(\nabla_h \bm{u}, \nabla_h \bm{v})+a^{i}_{h}(\bm{u},\bm{v})+a^{\partial}_{h}(\bm{u},\bm{v}),\\
a^{i}_{h}(\bm{u},\bm{v})&=a^{i}_p(\bm{u},\bm{v})-a^{i}_c(\bm{u},\bm{v})-a^{i}_c(\bm{v},\bm{u}),\\
a^{\partial}_{h}(\bm{u},\bm{v})&=a^{\partial}_p(\bm{u},\bm{v})-a^{\partial}_c(\bm{u},\bm{v})-a^{\partial}_c(\bm{v},\bm{u}),
\end{align}
where the interior face terms $a^{i}_p(\bm{u},\bm{v})$, $a^{i}_c(\bm{u},\bm{v})$ and Nitsche terms 
are defined as
\begin{align*}
a^{i}_c(\bm{u},\bm{v})&=\frac{\nu}{2}\sum_{E\in\mathcal{E}^i_h}\int_E [\![\nabla \bm{u}]\!]\!:\![\![\bm{v}\otimes \bm{n}]\!]d\bm{s}, 
\quad a^{i}_p(\bm{u},\bm{v})=\nu\sum_{E\in\mathcal{E}^i_h}\int_E\gamma_h [\![\bm{u}\otimes \bm{n}]\!]\!:\![\![\bm{v}\otimes \bm{n}]\!]d\bm{s},\\
a^{\partial}_c(\bm{u},\bm{v})&=\nu\sum_{E\in\mathcal{E}^{\partial}_h}\int_E \nabla \bm{u}\!:\!(\bm{v}\otimes \bm{n})d\bm{s}, 
\quad a^{\partial}_p(\bm{u},\bm{v})=2\nu\sum_{E\in\mathcal{E}^{\partial}_h}\int_E\gamma_h (\bm{u}\otimes \bm{n})\!:\!(\bm{v}\otimes \bm{n})d\bm{s},
\end{align*}
for $\gamma_h=\frac{\gamma}{h_E}$, and $\bm{u},\bm{v}\in \bm{V}_h$. 
Here, $h_E$ is the length of the edge $E$ and $\gamma$ is the penalty parameter which is chosen sufficiently large to guarantee the stability of the DG formulation, see for instance \cite{ADN}.
\par
Finally, we introduce the following mesh-dependent DG velocity-pressure norm
\begin{align}\label{desnorm}
\vertiii{(\bm{u},p)}^2=\vertiii{\bm{u}}^2+\nu^{-1}\|p\|_{0}^2,
\end{align}
where
\begin{align*}
\vertiii{\bm{u}}^2&=\nu\|\nabla_h\bm{u}\|^2_{0}+a^{i}_p(\bm{u},\bm{u})+a^{\partial}_p(\bm{u},\bm{u}).
\end{align*}

\section{A priori error analysis}\label{sec:2}
Our main aim is to show that the approximated eigenvalues and eigenfunctions of the $H^{\textrm{div}}$-DG  finite element formulation of the Stokes eigenvalue problem
converge to the solution of the corresponding spectral problem which comes to apply the classical spectral approximation theory presented in \cite{IBJO,Mercier}
using results of the a priori error analysis of the associated source problem that we recall here for completeness.

\subsection{Numerical analysis of the source problem}\label{numsp}
This section is devoted to discuss the source problem and to recall its essential stability and convergent results.
\par
Consider the source problem with the right hand side $\bm{f}\in \bm{L}^2(\Omega)$
\begin{align*}
-\nu \bigtriangleup \bm{u}^{\bm{f}}+\nabla p^{\bm{f}} &=\bm{f}\quad \mbox{in}\;\Omega,  \\
\nabla\cdot \bm{u}^{\bm{f}}&=0\quad \;\;\;\mbox{in}\;\Omega,\nonumber\\
 \bm{u}^{\bm{f}}&=0\quad \;\;\;\mbox{on}\;\partial\Omega,\nonumber
\end{align*}
with compatibility condition
\begin{align*}
\int_{\Omega} p\,dx=0.
\end{align*}
The variational formulation of the Stokes source problem reads: find $(\bm{u}^{\bm{f}},p^{\bm{f}})\in\bm{H}_0^1(\Omega)\times L^2_0(\Omega)$  such that
\begin{align}\label{sprob131}
\mathcal{A}(\bm{u}^{\bm{f}},p^{\bm{f}};\bm{v},q)=(\bm{f},\bm{v})\quad \forall (\bm{v},q)\in\bm{H}_0^1(\Omega)\times L^2_0(\Omega).
\end{align}
Due to the continuous inf-sup condition
\begin{align}\label{inf-sup}
\inf_{0\neq q\in L^2_0(\Omega)}\sup_{\bm{0}\neq\bm{v}\in\bm{H}^1_0(\Omega)}\frac{- (q,\nabla\cdot\bm{v})}{\|\nabla\bm{v}\|_{0}\|q\|_0}>0,
\end{align}
the variational formulation \eqref{sprob131} is well-posed \cite{BF1991,GR1986}.
\par
The  $H^{\textrm{div}}$-DG  finite elements of the Stokes source problem reads: find $(\bm{u}_h^{\bm{f}},p_h^{\bm{f}})\in\bm{V}_h\times Q_h$  such that
\begin{align}\label{sprob13}
\mathcal{A}_h(\bm{u}_h^{\bm{f}},p_h^{\bm{f}};\bm{v}_h,q_h)=(\bm{f},\bm{v}_h)\quad \forall (\bm{v}_h,q_h)\in\bm{V}_h\times Q_h.
\end{align}
\par
From \cite{BCGKDS,GKDS,GKNS}, we have that
the bilinear form $a_h(\cdot,\cdot)$ is bounded and elliptic uniformly in $h$ on $\bm{V}_h$ equipped with the norm $\vertiii{\cdot}$. 
Furthermore, the velocity-pressure pair $\bm{V}_h\times Q_h$ is inf-sup stable and satisfies
\begin{align*}
	\inf_{0\neq q_h\in Q_h}\sup_{\bm{0}\neq \bm{v}_h\in\bm{V}_h} \frac{-(q_h,\nabla\cdot\bm{v}_h)}{\vertiii{\bm{v}_h}\|q_h\|_0}\geq \beta > 0,
\end{align*}
for a constant $\beta$ independent of $h$.
Hence, the weak formulation \eqref{sprob13} has a 
unique discrete solution, which admits the following stability estimate
\begin{align*}
\vertiii{(\bm{u}_h^{\bm{f}},p_h^{\bm{f}})}&\lesssim \nu^{-1/2} \|\bm{f}\|_0,
\end{align*} 
and due to $\nabla\cdot\bm{V}_h\subset Q_h$ the discrete velocity $\bm{u}_h^{\bm{f}}$ is exactly divergence-free.
\par
From \cite{BCGKDS,GVGKBR,GKDS}, we have the following a priori estimates.
\begin{theorem}\label{errapriori1}
Let $\bm{u}\in\bm{H}^{s}(\Omega)$ and  $p\in{H}^{s-1}(\Omega)$ for some $s\in ]3/2,k+1]$ be solutions of 
the continuous problem \eqref{sprob131} that satisfy the following regularity condition 
\begin{align}\label{regularity}
 \nu^{1/2}\|\bm{u}^{\bm{f}}\|_{s}+\nu^{-1/2}\|p^{\bm{f}}\|_{s-1}\lesssim 
 \left\{
  \begin{array}{lr}
   \nu^{-1/2}\|\bm{f}\|_{0}, &  s\in(\frac{3}{2},2),\\
  \nu^{-1/2} \|\bm{f}\|_{s-2}, &  s\ge 2.
  \end{array}
\right.
 \end{align}
Then we have the following error bounds for the discrete approximations  $(\bm{u}_h,p_h)\in\bm{V}_h\times Q_h$
of the discrete problem \eqref{sprob13}
\begin{align}\label{errapriori}
\vertiii{(\bm{u}^{\bm{f}}-\bm{u}_h^{\bm{f}},p^{\bm{f}}-p_h^{\bm{f}})}&\lesssim h^{s-1}(\nu^{1/2}\|\bm{u}^{\bm{f}}\|_{s}+\nu^{-1/2}\|p^{\bm{f}}\|_{s-1}),\\
\|\bm{u}^{\bm{f}}-\bm{u}_h^{\bm{f}}\|_0&\lesssim h^{s-1+\alpha}\|\bm{u}^{\bm{f}}\|_{s},\label{errapriori2}
\end{align}
for $\alpha=\min\{s-1,1\}$. \qquad
\end{theorem}

\subsection{Numerical analysis of the eigenvalue problem}\label{numep}
We now apply the Babu\v{s}ka-Osborn theory to derive the 
convergence of eigenvalues and eigenfunctions of the discrete problem \eqref{hdiv11} to those of the continuous problem \eqref{stpre11}
and estimate the order of convergence. 
\par
Using the well posedness of the continuous source problem \eqref{sprob131}, the operators $T:\bm{L}^2(\Omega)\rightarrow\bm{H}^1_0(\Omega)$ and $S:\bm{L}^2(\Omega)\rightarrow L^2_0(\Omega)$
are well defined for any $\bm{f}\in\bm{L}^2(\Omega)$ such that
$T\bm{f}=\bm{u}^{\bm{f}}$ and $S\bm{f}=p^{\bm{f}}$ are the velocity and pressure components of the solution to problem \eqref{sprob131}.
\par
Since the discrete source problem \eqref{sprob13} is well posed, we define in the same manner the operators
$T_h:\bm{L}^2(\Omega)\rightarrow\bm{V}_h$ and $S_h:\bm{L}^2(\Omega)\rightarrow Q_h$ such that
$T_{h}\bm{f}=\bm{u}_h^{\bm{f}}$ and $S_h\bm{f}=p_h^{\bm{f}}$ are the discrete velocity and the discrete pressure approximations.
Note that the operator $T_h$ is well defined in $\bm{L}^2(\Omega)$ but not in $\bm{H}^1_0(\Omega)$.
Hence, we can only conclude convergence of the operators $T_h$ in $\bm{L}^2(\Omega)$ from the abstract theory.
\par
From the a priori estimates \eqref{errapriori} and \eqref{errapriori2} for the soure problem,
we conclude
\begin{align}\label{convere11}
\|T-T_h\|_{\mathcal{L}(\bm{L}^2(\Omega), \bm{L}^2(\Omega))}&\lesssim h^{s-1+\alpha},\\
\|S-S_h\|_{\mathcal{L}(\bm{L}^2(\Omega), L^2_0(\Omega))}&\lesssim h^{s-1},
\end{align}
which leads to the convergence of eigenvalues and eigenfunctions.
From the Babu\v{s}ka-Osborn theory we get the following rates of convergence
for the $\bm{L}^2(\Omega)$ error for the velocity component $\bm{u}$ and the $L^2(\Omega)$ error for the pressure component $p$ of eigenfunctions
under the regularity condition \eqref{regularity}
\begin{align}\label{eigerr1}
\|\bm{u}-\bm{u}_h\|_0&\lesssim h^{s-1+\alpha} \|\bm{u}\|_{s},\\
\|p-p_h\|_0&\lesssim h^{s-1}(\nu^{1/2}\|\bm{u}\|_{s}+\nu^{-1/2}\|p\|_{s-1}), \nonumber
\end{align}
for $s\in]3/2,k+1]$, and $\alpha=\min\{s-1,1\}$.
From Mercier et al. \cite{Mercier}, we conclude that the eigenvalues converge twice as fast as the eigenfunctions, i.e. 
\begin{align}\label{eigenrr11}
|\lambda-\lambda_h|\lesssim h^{2(s-1)}\quad\text{for } s\in]3/2,k+1].
\end{align}
\par
\begin{theorem}\label{eigerr2}
The following estimate holds for $s\in]3/2,k+1]$
\begin{align*}
\vertiii{\bm{u}-\bm{u}_h}\lesssim h^{s-1}(h^{s-1}+h^{\alpha}\|\bm{u}\|_{s}+\nu^{1/2}\|\bm{u}\|_{s}+\nu^{-1/2}\|p\|_{s-1}).
\end{align*}
\end{theorem}
\begin{proof}
Let $\lambda$ be the eigenvalue corresponding to the eigenfunction $\bm{u}$. Then it holds that
\begin{align*}
\bm{u}-\bm{u}_h&=\lambda T\bm{u}-\lambda_h T_h\bm{u}_h
=(\lambda-\lambda_h)T\bm{u}+\lambda_h(T-T_h)\bm{u}+\lambda_h T_h(\bm{u}-\bm{u}_h).
\end{align*}
It follows that
\begin{align*}
\vertiii{\bm{u}-\bm{u}_h}\lesssim |\lambda-\lambda_h| \|T\bm{u}\|_{1}+\lambda_h\vertiii{(T-T_h)\bm{u}}+\lambda_h\vertiii{T_h(\bm{u}-\bm{u}_h)}.
\end{align*}
The first two terms of the right hand side are directly estimated by \eqref{eigenrr11} and \eqref{errapriori},
and the last term is estimated using \eqref{sprob13} as follows
\begin{align*}
\vertiii{T_h(\bm{u}-\bm{u}_h)}^2\lesssim \mathcal{A}_h(T_h(\bm{u}-\bm{u}_h),p_h;T_h(\bm{u}-\bm{u}_h),p_h)=((\bm{u}-\bm{u}_h),T_h(\bm{u}-\bm{u}_h)).
\end{align*}  
Applying Cauchy-Schwartz inequality and \eqref{eigerr1}, implies
\begin{align*}
\vertiii{T_h(\bm{u}-\bm{u}_h)}\lesssim \|\bm{u}-\bm{u}_h\|_0\lesssim h^{s-1+\alpha} \|\bm{u}\|_{s},
\end{align*}
where $\alpha=\min\{s-1,1\}$. 
\end{proof}
\par
We will now establish a relationship between the eigenvalue and the eigenfunction errors.
In order to do so, we observe that the numerical scheme is consistent.
\begin{lemma}\label{consis11}
Let  $(\bm{u},p,\lambda)\in \bm{H}^1_0(\Omega)\times L^2_0(\Omega)\times \mathbb{R}_+$ be the solution of \eqref{stpre11}.
If $\bm{u}\in \bm{H}^2(\mathcal{T}_h)$ and $p\in H^{1}(\mathcal{T}_h)$, then
\begin{align*}
\mathcal{A}_h(\bm{u},p;\bm{v}_h,q_h) = \lambda(\bm{u},\bm{v}_h)\quad\forall (\bm{v}_h,q_h)\in \bm{V}_h\times Q_h.
\end{align*}
\end{lemma}
\begin{proof}
The result follows from the consistency of the discontinuous Galerkin finite element method for the source problem \cite[Lemma 7.5]{AT}. 
\end{proof}
\begin{theorem}\label{riziden11}
Let  $(\bm{u},p,\lambda)\in \bm{H}^1_0(\Omega)\times L^2_0(\Omega)\times \mathbb{R}_+$ be the solution of \eqref{stpre11}
and $(\bm{u}_h,p_h)\in \bm{V}_h\times Q_h$  with $\|\bm{u}_h\|_0\neq0$. 
{If $\bm{u}\in \bm{H}^2(\mathcal{T}_h)$ and $p\in H^{1}(\mathcal{T}_h)$,} then the Rayleigh quotient satisfies the following identity 
\begin{align*}
\frac{\mathcal{A}_h(\bm{u}_h,p_h;\bm{u}_h,p_h)}{\|\bm{u}_h\|_{0}^2}-\lambda=& \frac{\mathcal{A}_h(\bm{u}-\bm{u}_h,p-p_h;\bm{u}-\bm{u}_h,p-p_h)}{\|\bm{u}_h\|_{0}^2}-\lambda\frac{\|\bm{u}-\bm{u}_h\|^2_{0}}{\|\bm{u}_h\|_{0}^2}.
\end{align*}
\end{theorem}
\begin{proof}
Note that 
\begin{align*}
\mathcal{A}_h(\bm{u},p;\bm{u},p)=\mathcal{A}(\bm{u},p;\bm{u},p)\quad\forall \;(\bm{u},p)\in \bm{H}^{1}_0(\Omega)\times L^2_0(\Omega).
\end{align*}
Moreover, from consistency we get
\begin{align}\label{eigenval11}
\begin{split}
&\mathcal{A}_h(\bm{u}-\bm{u}_h,p-p_h;\bm{u}-\bm{u}_h,p-p_h)\\
&\qquad=\mathcal{A}(\bm{u},p;\bm{u},p)+\mathcal{A}_h(\bm{u}_h,p_h;\bm{u}_h,p_h)
-2\mathcal{A}_h(\bm{u}_h,p_h;\bm{u},p)\\
&\qquad=\lambda(\bm{u},\bm{u})+\mathcal{A}_h(\bm{u}_h,p_h;\bm{u}_h,p_h)
- 2\lambda(\bm{u},\bm{u}_h).
\end{split}
\end{align}
Next we write the following identity
\begin{align}\label{eigenval112}
\lambda(\bm{u}-\bm{u}_h,\bm{u}-\bm{u}_h)=\lambda(\bm{u},\bm{u})+\lambda(\bm{u}_h,\bm{u}_h)-2\lambda(\bm{u},\bm{u}_h).
\end{align}
Subtracting \eqref{eigenval112} from \eqref{eigenval11}, we obtain
\begin{align*}
&\mathcal{A}_h(\bm{u}-\bm{u}_h,p-p_h;\bm{u}-\bm{u}_h,p-p_h)-\lambda(\bm{u}-\bm{u}_h,\bm{u}-\bm{u}_h)\\
&\qquad=\mathcal{A}_h(\bm{u}_h,p_h;\bm{u}_h,p_h)-\lambda(\bm{u}_h,\bm{u}_h).
\end{align*}
Dividing by $(\bm{u}_h,\bm{u}_h)$ on both sides in the above equation ends the proof. 
\end{proof}

\section{A posteriori error analysis}\label{aposterrestana}
In this section, we present a residual-based 
a posteriori error estimator for the Stokes eigenvalue problem.
\par
Let $(\bm{u}_h,p_h,\lambda_h)\in \bm{V}_h\times Q_h \times \mathbb{R}_+$ be an eigentriple approximation.
For each $K\in \mathcal{T}_h$, 
 the interior residual estimator $\eta_{R_K}$ is defined by
 \begin{align*}
 \eta_{R_K}^2 :=\nu^{-1}h^2_K\|\lambda_h \bm{u}_h+\nu\Delta\bm{u}_h -\nabla p_h\|_{0,K}^2,
 \end{align*}
and the edge residual estimator $\eta_{E_K}$ by
\begin{align*}
\eta_{E_K}^2 :=\nu^{-1}\sum_{E\in \partial K\setminus\partial\Omega}h_E \| [\![(p_h\bm{I}-\nu\nabla \bm{u}_h)\bm{n}]\!]\|^{2}_{0,E},
\end{align*}
where $\bm{I}$ denotes the $2\times2$ identity matrix.
Next, we introduce the estimator $\eta_{J_K}$, which measures the jump of the approximate solution $\bm{u}_h$,
\begin{align*}
\eta_{J_K}^2 :=\nu \sum_{E\in \partial K} \gamma_h \|[\![\bm{u}_h\otimes\bm{n}]\!]\|^{2}_{0,E}.
\end{align*}
The local error indicator, which is the sum of the above three terms, is defined as
\begin{align*}
\eta_K^2 :=\eta_{R_K}^2+\eta_{E_K}^2+\eta_{J_K}^2.
\end{align*}
 Finally, we introduce the (global) a posteriori error estimator
 \begin{align}\label{errest1}
 \eta_h :=\Big(\sum_{K\in\mathcal{T}_h}\eta_K^2\Big)^{1/2}.
 \end{align}

\subsection{Additional stability property}
In the proof of reliability we will use the following auxiliary stability property following \cite[Lemma 4.3]{HSW}, \cite[Section 2.3]{GKDS}. 
We include the proof for the $H^{\textrm{div}}$-DG formulation of the Stokes problem for completeness.
\begin{lemma}\label{Sinsuplem12}
For any $(\bm{u},p)\in\bm{H}^{1}_0(\Omega)\times L^2_0(\Omega)$, there exists a pair $(\bm{v},q)\in \bm{H}^{1}_0(\Omega)\setminus\{0\}\times L^2_0(\Omega)$ 
with $\vertiii{(\bm{v},q)}\lesssim \vertiii{(\bm{u},p)}$ and
\begin{align*}
 \mathcal{A}_h(\bm{u},p; \bm{v},q)\gtrsim \vertiii{(\bm{u},p)}^2.
\end{align*}
\end{lemma}
{\begin{proof}
From the continuous inf-sup condition \eqref{inf-sup} we deduce that there exists a $\bm{w}\in\bm{H}^1_0(\Omega)$ such that 
\begin{align*}
-(p,\nabla\cdot\bm{w})\ge C_{\Omega}\nu^{-1} \|p\|^2_{0},
\qquad\text{and}\qquad
 \nu^{1/2}\|\nabla\bm{w}\|_0\le\nu^{-1/2} \|p\|_{0},
\end{align*}
where $C_\Omega>0$ is the continuous inf-sup constant, which only depends on $\Omega$.
If $(\bm{u},p)\in \bm{H}^{1}_0(\Omega)\times L^2_0(\Omega)$, then it holds that
\begin{align}\label{infsup1}
\mathcal{A}_h(\bm{u},p;\bm{u},-p)\ge \vertiii{\bm{u}}^2.
\end{align}
 Moreover, using $\nu^{1/2}\|\nabla\bm{w}\|_0\le\nu^{-1/2} \|p\|_{0}$, continuity and Youngs inequality we get
\begin{align}\label{infsup2}
\begin{split}
\mathcal{A}_h(\bm{u},p;\bm{w},0)&\ge C_{\Omega}\nu^{-1}\|p\|^2_{0}-C\vertiii{\bm{u}}\;\vertiii{\bm{w}}\\
&\ge C_{\Omega}\nu^{-1}\|p\|^2_{0}-\nu^{-1/2}C\vertiii{\bm{u}}\;\|p\|_{0}\\
&\ge \Bigg(C_{\Omega}-\frac{1}{\epsilon}\Bigg)\nu^{-1}\|p\|^2_{0}-\epsilon  C^2\vertiii{\bm{u}}^2,
\end{split}
\end{align} 
for a positive generic continuity constant $C> 0$.
Using equations \eqref{infsup1} and \eqref{infsup2}, we have
\begin{align*}
\mathcal{A}_h(\bm{u},p;\bm{u}+\delta \bm{w},-p)&=\mathcal{A}_h(\bm{u},p;\bm{u},-p)+\delta\mathcal{A}_h(\bm{u},p;\bm{w},0)\\
&\ge \vertiii{\bm{u}}^2+\delta \Bigg(C_{\Omega}-\frac{1}{\epsilon}\Bigg)\nu^{-1}\|p\|^2_{0}-\delta\epsilon C^2\vertiii{\bm{u}}^2\\
&\ge (1-\delta\epsilon C^2)\vertiii{\bm{u}}^2+\delta \Bigg(C_{\Omega}-\frac{1}{\epsilon}\Bigg)\nu^{-1}\|p\|^2_{0}.
\end{align*}
Taking $\epsilon =2/C_\Omega$ and $\delta=C_\Omega/(4 C^2)$, it follows
\begin{align}\label{infsup11}
\mathcal{A}_h(\bm{u},p;\bm{u}+\delta \bm{w},-p)&\ge \min\Bigg\{\frac{1}{2},\frac{C_\Omega^2}{8 C^2}\Bigg\}\vertiii{(\bm{u},p)}^2.
\end{align}
Moreover, from $\nu^{1/2}\|\nabla\bm{w}\|_0\le \nu^{-1/2}\|p\|_{0}$ we get
\begin{align}\label{infsup21}
\begin{split}
\vertiii{(\bm{u}+\delta\bm{w},-p)}^2
&\le 2\vertiii{\bm{u}}^2 + 2\delta^2\nu\|\nabla\bm{w}\|_0^2 + \nu^{-1}\|p\|_0^2\\
&\le \max\left\{2,\left(1+\frac{C_\Omega^2}{8 C^4}\right)\right\}\vertiii{(\bm{u},p)}^2.
\end{split}
\end{align}
Combining equations \eqref{infsup11} and  \eqref{infsup21}, proves the final assertion with $\bm{v}=\bm{u}+\delta\bm{w}$ and $q=-p$. 
\end{proof}}

\subsection{Reliability}
First we define the discontinuous $RT_k$ space $\widetilde{\bm{V}}_h=\{\bm{v}\in \bm{L}^2:\bm{v}|_K\in RT_{k}(K),\, K\in\mathcal{T}_h\}$. 
As in \cite{HSW,GKDS}, we define ${\bm{V}}_h^c=\widetilde{\bm{V}}_h\cap \bm{H}^{1}_{0}(\Omega)$. 
The orthogonal complement of ${\bm{V}}_h^c$ in $\widetilde{\bm{V}}_h$ with respect to the norm $\vertiii{\cdot}$ 
is defined by ${\bm{V}}_h^{\perp}$. Then we obtain $\widetilde{\bm{V}}_h={\bm{V}}_h^c\oplus{\bm{V}}_h^{\perp}$.
 Hence, we decompose the DG velocity approximation uniquely into
 \begin{align*}
 \bm{u}_h=\bm{u}_h^c+\bm{u}_h^r,
\end{align*}  
where $\bm{u}_h^c\in {\bm{V}}_h^c$ and $\bm{u}_h^r\in {\bm{V}}_h^{\perp}$.
Using the triangle inequality, we can write
\begin{align}
\vertiii{\bm{u}-\bm{u}_h}\le  \vertiii{\bm{u}-\bm{u}_h^c}+\vertiii{\bm{u}_h^r},
\end{align}
and from \cite[Proposition 4.1]{HSW} we get the upper bound for the second term
\begin{align}\label{realilem1}
\vertiii{\bm{u}_h^r}\lesssim \Big(\sum_{K\in\mathcal{T}_h}\eta_{J_K}^2\Big)^{1/2}.
\end{align}
\par
Note that the DG form  $a_{h}(\bm{u},\bm{v})$ is not well defined for functions $\bm{u},\bm{v}$ which belong to $\bm{H}^{1}_0(\Omega)$. One can overcome this difficulty by the use of a suitable lifting operator, cf. \cite{BCGKDS,GKDS}. Here, we discuss a different approach where the DG form $a_{h}(\cdot,\cdot)$ is split into several parts,
\begin{align*}
a_{h}(\bm{u},\bm{v}) = \nu(\nabla_h \bm{u}, \nabla_h\bm{v}) + C_{h}(\bm{u},\bm{v}) + J_h(\bm{u},\bm{v}),
\end{align*}
with
 \begin{align*}
 C_{h}(\bm{u},\bm{v})&=-a^{i}_c(\bm{u},\bm{v})-a^{i}_c(\bm{v},\bm{u})-a^{\partial}_c(\bm{u},\bm{v})-a^{\partial}_c(\bm{v},\bm{u}),\\
 J_h(\bm{u},\bm{v})&=a^{i}_p(\bm{u},\bm{v})+a^{\partial}_p(\bm{v},\bm{u}).
\end{align*}
\begin{lemma}\label{aplem11}
Let $\bm{u}_h\in \bm{V}_h$ and $\bm{v}_h^c\in \bm{V}_h^c$, then it holds that
\begin{align*}
C_{h}(\bm{u}_h,\bm{v}_h^c)\lesssim \gamma^{-1/2}\Big(\sum_{K\in\mathcal{T}_h}\eta_{J_K}^2\Big)^{1/2}\vertiii{\bm{v}}.
\end{align*}
\end{lemma}
\begin{proof}
Since $\bm{v}_h^c\in \bm{V}_h^c$, we have
\begin{align*}
C_{h}(\bm{u}_h,\bm{v}_h^c)=-a^{i}_c(\bm{v}_h^c,\bm{u}_h)-a^{\partial}_c(\bm{v}_h^c,\bm{u}_h).
\end{align*} 
Applying Cauchy-Schwarz inequality, implies 
\begin{align*}
C_{h}(\bm{u}_h,\bm{v}_h^c)\lesssim \Big(\nu\sum_{E\in\mathcal{E}_h} \gamma_h^{-1} \|[\![\nabla\bm{v}_h^c]\!]\|_{0,E}^2\Big)^{1/2}
\,\Big(\nu\sum_{E\in\mathcal{E}_h} \gamma_h \|[\![\bm{u}_h\otimes\bm{n}]\!]\|_{0,E}^2\Big)^{1/2}.
\end{align*} 
Using a trace estimate together with a discrete inverse inequality leads for an edge $E\in\mathcal{E}_h$, with $E=K_1\cap K_2$ if $E\in\mathcal{E}_h^i$ and
$E = K_1$, $K_2=\emptyset$ if $E\subset\partial\Omega$, to
\begin{align*}
\|\nabla\bm{v}_h^c\|_{0,E}\lesssim h_{K}^{-1/2}\|\nabla\bm{v}_h^c\|_{0,K_1\cup K_2}.
\end{align*}
Thus we have
\begin{align*}
C_{h}(\bm{u}_h,\bm{v}_h^c)\lesssim\gamma^{-1/2} \Big(\nu\sum_{K\in\mathcal{T}_h}\|\nabla\bm{v}_h^c\|^2_{0,K}\Big)^{1/2}\,\Big(\sum_{K\in\mathcal{T}_h}\eta_{J_K}^2\Big)^{1/2}. 
\end{align*}
\end{proof}
\par
Let $\bm{\Pi}_h : \bm{H}^1_0 \to \bm{V}_h^c$ denote the Scott-Zhang interpolation operator \cite{SZ}, which is stable
$\|\nabla(\bm{\Pi}_h\bm{v})\|_0\lesssim\|\nabla\bm{v}\|_0$
and satisfies the following
interpolation property  
 \begin{align}\label{approxlem12}
\sum_{K\in\mathcal{T}_h} h_{K}^{-2}\|\bm{v}-\bm{\Pi}_h\bm{v}\|^{2}_{0,K} + 
\sum_{E\in\mathcal{E}_h} h_{E}^{-1}\|\bm{v}-\bm{\Pi}_h\bm{v}\|^{2}_{0,E}&\lesssim \|\nabla\bm{v}\|_0^2,
 \end{align}
 for any $\bm{v}\in\bm{H}^{1}_0(\Omega)$.
\begin{lemma}\label{reallem123}
Let $\bm{v}_h^c=\bm{\Pi}_h\bm{v}\in\bm{V}_h^c$ be the Scott-Zhang interpolation of $\bm{v}\in \bm{H}^1_0(\Omega)$, then for any $\bm{u}_h\in \bm{V}_h$, $p_h\in Q_h$, and $\lambda_h\in\mathbb{R}_+$, it holds that
\begin{align}
\lambda_h(\bm{u}_h,\bm{v}-\bm{v}_h^c)-\nu(\nabla_h\bm{u}_h,\nabla(\bm{v}-\bm{v}_h^c))+(p_h,\nabla\cdot(\bm{v}-\bm{v}_h^c))
\lesssim \eta_h\vertiii{\bm{v}}.
\end{align}
\end{lemma}
\begin{proof}
Using integration by parts on each element $K\in\mathcal{T}_h$, we have
\begin{align*}
&\lambda_h(\bm{u}_h,\bm{v}-\bm{v}_h^c)-\nu(\nabla_h\bm{u}_h,\nabla(\bm{v}-\bm{v}_h^c))+(p_h,\nabla\cdot(\bm{v}-\bm{v}_h^c))\\
&\qquad=\sum_{K\in\mathcal{T}_h}\int_{K}(\lambda_h\bm{u}_h+\nu\Delta \bm{u}_h-\nabla p_h)(\bm{v}-\bm{v}_h^c)d\bm{x}\\
&\qquad\qquad
+\sum_{K\in\mathcal{T}_h}\int_{\partial K}(p_h\bm{I} - \nu\nabla{\bm{u}}_h)\bm{n}_K\cdot(\bm{v}-\bm{v}_h^c)d\bm{s}\\
&\qquad=T_1+T_2.
\end{align*}
Cauchy-Schwarz inequality and \eqref{approxlem12}, lead to
\begin{align*}
T_1&\lesssim\Big(\nu^{-1}\sum_{K\in\mathcal{T}_h}h_K^2\|\lambda_h\bm{u}_h+\nu\Delta \bm{u}_h-\nabla p_h\|_{0,K}^2\Big)^{1/2}
\Big(\nu\sum_{K\in\mathcal{T}_h}h_{K}^{-2}\|\bm{v}-\bm{v}_h^c\|^{2}_{0,K}\Big)^{1/2}\\
&\lesssim\Big(\sum_{K\in\mathcal{T}_h}\eta_{R_K}^2\Big)^{1/2}\vertiii{\bm{v}}.
\end{align*}
Since $(\bm{v}-\bm{v}_h^c)|_{\partial\Omega}=0$ we can rewrite $T_2$ in terms of a sum over interior edges
\begin{align*}
T_2= \sum_{E\in\mathcal{E}^i_h}\int_E[\![(p_h\bm{I}-\nu\nabla\bm{u}_h)\bm{n}]\!](\bm{v}-\bm{v}_h^c)d\bm{s}.
\end{align*}
Again, applying Cauchy-Schwarz inequality and \eqref{approxlem12}, imply
\begin{align*}
T_2&\lesssim\Big(\nu^{-1}\sum_{E\in\mathcal{E}^i_h}h_{E}\|[\![(p_h\bm{I}-\nu\nabla\bm{u}_h)\bm{n}]\!]\|^{2}_{0,E}\Big)^{1/2}\Big(\nu\sum_{E\in\mathcal{E}^i_h}
h_{E}^{-1}\|\bm{v}-\bm{v}_h^c\|^{2}_{0,E}\Big)^{1/2}\\
&\lesssim \Big(\sum_{K\in\mathcal{T}_h}\eta_{E_K}^2\Big)^{1/2}\vertiii{\bm{v}}.
\end{align*}
Combining the above estimates, proves the desired result. 
\end{proof}

\begin{lemma}\label{lemma5}
Let $(\bm{u},p,\lambda)\in \bm{H}^1_0(\Omega)\times L^2_0(\Omega) \times \mathbb{R}_+$ solve \eqref{stpre11}
and $(\bm{u}_h,p_h,\lambda_h)\in \bm{V}_h\times Q_h \times \mathbb{R}_+$ solve \eqref{hdiv11}, then we have the following upper bound
for the conforming velocity and pressure errors
\begin{align*}
\vertiii{\bm{u}-\bm{u}_h^c}+\nu^{-1/2}\|p-p_h\|_{0}\lesssim \eta_h + \nu^{-1/2}\left(|\lambda-\lambda_h|+\lambda\|\bm{u}-\bm{u}_h\|_0\right).
\end{align*}
\end{lemma}
\begin{proof}
Using Lemma \ref{Sinsuplem12}, there exists a pair $(\bm{v},q)\in \bm{H}^{1}_0(\Omega)\setminus\{0\}\times L^2_0(\Omega)$ such that
\begin{align*}
\vertiii{(\bm{u}-\bm{u}_h^c,p-p_h)}^2\lesssim {\mathcal{A}}_h(\bm{u}-\bm{u}_h^c,p-p_h;\bm{v},q),
\end{align*}
and
\begin{align*}
\vertiii{(\bm{v},q)}\lesssim \vertiii{(\bm{u}-\bm{u}_h^c,p-p_h)}.
\end{align*}
Since $\bm{u},\bm{u}_h^c,\bm{v}\in \bm{H}^1_0(\Omega)$, we have
\begin{align*}
{\mathcal{A}}_h(\bm{u}-\bm{u}_h^c,p-p_h;\bm{v},q)=\nu(\nabla(\bm{u}-\bm{u}_h^c),\nabla\bm{v})
-(p-p_h,\nabla\cdot\bm{v})-(q,\nabla\cdot(\bm{u}-\bm{u}_h^c)).
\end{align*}
From \eqref{stpre11}, we obtain
\begin{align*}
{\mathcal{A}}_h(\bm{u}-\bm{u}_h^c,p-p_h;\bm{v},q)=\lambda(\bm{u},\bm{v})
-\nu(\nabla\bm{u}_h^c,\nabla\bm{v})
+(p_h,\nabla\cdot\bm{v})+(q,\nabla\cdot\bm{u}_h^c).
\end{align*}
Applying the fact  $(q,\nabla\cdot\bm{u}_h)=0$, implies
\begin{align*}
{\mathcal{A}}_h(\bm{u}-\bm{u}_h^c,p-p_h;\bm{v},q)&=\lambda(\bm{u},\bm{v})
-\nu(\nabla\bm{u}_h^c,\nabla\bm{v})
+(p_h,\nabla\cdot\bm{v})-(q,\nabla\cdot\bm{u}_h^r)\\
&= \lambda_h(\bm{u}_h,\bm{v})+(\lambda\bm{u}-\lambda_h\bm{u}_h,\bm{v})
-\nu(\nabla_h\bm{u}_h,\nabla\bm{v})\\
&\quad+\nu(\nabla_h\bm{u}_h^r,\nabla\bm{v})
+(p_h,\nabla\cdot\bm{v})-(q,\nabla\cdot\bm{u}_h^r).
\end{align*}
Let $\bm{v}_h^c=\bm{\Pi}_h\bm{v}\in\bm{V}_h^c$ be the Scott-Zhang interpolation of $\bm{v}$.
Using 
\[0=\lambda_h(\bm{u}_h,\bm{v}_h^c)-\nu(\nabla_h\bm{u}_h,\nabla\bm{v}_h^c)-C_h(\bm{u}_h,\bm{v}_h^c)+(p_h,\nabla\cdot\bm{v}_h^c),\] yields
\begin{align*}
{\mathcal{A}}_h(\bm{u}-\bm{u}_h^c,p-p_h;\bm{v},q)&=T_1+T_2+T_3+T_4,
\end{align*}
where
\begin{align*}
T_1&=\lambda_h(\bm{u}_h,\bm{v}-\bm{v}_h^c)-\nu(\nabla_h\bm{u}_h,\nabla(\bm{v}-\bm{v}_h^c))+(p_h,\nabla\cdot(\bm{v}-\bm{v}_h^c)),\\
T_2&= \nu(\nabla_h\bm{u}_h^r,\nabla\bm{v})-(q,\nabla\cdot\bm{u}_h^r),\quad
T_3=C_h(\bm{u}_h,\bm{v}_h^c),\quad
T_4=(\lambda\bm{u}-\lambda_h\bm{u}_h,\bm{v}).
\end{align*}
Using Lemma \ref{reallem123}, we have
\begin{align*}
T_1&\lesssim \eta_h\vertiii{\bm{v}}.
\end{align*}
Cauchy-Schwarz inequality and \eqref{realilem1} show
\begin{align*}
T_2&\lesssim\vertiii{\bm{u}_h^r}\vertiii{(\bm{v},q)}\lesssim\eta_h\vertiii{(\bm{v},q)}.
\end{align*}
Using Lemma \ref{aplem11} for the bound of $T_3$, we have
\begin{align*}
T_3&\lesssim\gamma^{-1/2}\eta_h\vertiii{\bm{v}}.
\end{align*}  
Cauchy-Schwarz and Poincare inequality lead to
\begin{align*}
T_4&\lesssim \nu^{-1/2}\|\lambda\bm{u}-\lambda_h\bm{u}_h\|_{0} \vertiii{\bm{v}}
\lesssim  \nu^{-1/2}\left( |\lambda-\lambda_h| + \lambda\|\bm{u}-\bm{u}_h\|_0\right)\vertiii{\bm{v}}.
\end{align*} 
Combining the above with the estimate $\vertiii{(\bm{v},q)}\lesssim\vertiii{(\bm{u}-\bm{u}_h^c,p-p_h)}$ yields the desired result. 
\end{proof}

\begin{theorem}\label{realiab}
 Let $(\bm{u},p,\lambda)\in \bm{H}^1_0(\Omega)\times L^2_0(\Omega) \times \mathbb{R}_+$ be the solution of the Stokes eigenvalue problem \eqref{stpre11}  and $(\bm{u}_h,p_h,\lambda_h)\in\bm{V}_h\times Q_h$ the $H^{\textrm{div}}$-DG approximation obtained by \eqref{hdiv11}. Let $\eta_h$ be the a posteriori error estimator in \eqref{errest1}. Then we obtain the following a posteriori error bound
 \begin{align*}
 \vertiii{\bm{u}-\bm{u}_h}+\nu^{-1/2}\|p-p_h\|_{0}\lesssim \eta_h +  \nu^{-1/2}\left( |\lambda-\lambda_h| + \lambda\|\bm{u}-\bm{u}_h\|_0\right),
\end{align*}   
where the hidden constant is independent of the viscosity $\nu$ and the sufficiently large penalty parameter $\gamma\geq 1$.
\end{theorem}
\begin{proof}
  The proof follows directly from a combination of Lemma~\ref{lemma5} and \eqref{realilem1}. 
\end{proof}
 
 \begin{corollary}\label{cor11}
If $\bm{u}\in \bm{H}^2(\mathcal{T}_h)$ and $p\in H^{1}(\mathcal{T}_h)$,  then the eigenvalue error satisfies
 \begin{align*}
 |\lambda-\lambda_h|\lesssim \eta_h^2+\nu^{-1}|\lambda-\lambda_h|^2+ (\lambda+\nu^{-1}\lambda^2)\|\bm{u}-\bm{u}_h\|_0^2.
 \end{align*}
\end{corollary}
\begin{proof}
Note that since $\nabla\cdot \bm{u} = 0 = \nabla\cdot \bm{u}_h$, we have
\begin{align*}
\mathcal{A}_h(\bm{u}-\bm{u}_h,p-p_h;\bm{u}-\bm{u}_h,p-p_h)
&= a_h(\bm{u}-\bm{u}_h,\bm{u}-\bm{u}_h)\\
&= \vertiii{\bm{u}-\bm{u}_h}^2 + C_h(\bm{u}-\bm{u}_h,\bm{u}-\bm{u}_h).
\end{align*}
The consistency term can further be estimated as
\begin{align}
C_h(\bm{u}-\bm{u}_h,\bm{u}-\bm{u}_h) \le C_\ast |||(\bm{u}-\bm{u}_h,p-p_h)|||^2,\nonumber
 \end{align}
 where 
 \begin{align*}
 C_\ast= \frac{|C_h(\bm{u}-\bm{u}_h,\bm{u}-\bm{u}_h)|}{|||(\bm{u}-\bm{u}_h, p-p_h)|||^2}.
 \end{align*}
 From the estimates in \cite[Section 8]{DSCSAT}, we can conclude that $|C_h(\bm{u}-\bm{u}_h, \bm{u}-\bm{u}_h)|$
is of the same order as $|||(\bm{u}-\bm{u}_h, p-p_h)|||^2$. Hence $C_\ast$ can be bounded from above by a uniform constant.
  The assertion then follows from a combination of the above with Theorems \ref{riziden11} \& \ref{realiab}. 
\end{proof}
 
\subsection{Efficiency}
This section is devoted to prove an efficiency bound for $\eta$. To prove the results, we use the bubble function technique which was introduced in \cite{RV,RV1}. 
\par
Let $K$ be an element of $\mathcal{T}_h$.
We consider the standard element bubble function $b_K$ on $K$. Let $\bm{v}_h$ be any vector valued polynomial function on $K$,
then the following results hold from \cite{AMOJT,GKDS,RV},
\begin{align}
\begin{split}\label{efficie11}
\|\bm{v}_h\|_{0,K}&\lesssim \|b_K^{1/2}\bm{v}_h\|_{0,K},\\
\|b_K \bm{v}_h\|_{0,K}&\lesssim \|\bm{v}_h\|_{0,K},\\
\|\nabla(b_K\bm{v}_h)\|_{0,K}&\lesssim h_K^{-1}\|\bm{v}_h\|_{0,K}.
\end{split}
\end{align}

\begin{lemma}\label{efficiej11}
For $\bm{u}_h\in \bm{V}_h$, it holds that
\begin{align*}
\Big(\sum_{K\in\mathcal{T}_h}\eta_{J_K}^2\Big)^{1/2}\lesssim \vertiii{\bm{u}-\bm{u}_h}.
\end{align*}
\end{lemma}
\begin{proof}
Using $[\![\bm{u}\otimes\bm{n}]\!]=0$, we get 
\begin{align*}
\eta_{J_K}^2 =\nu\sum_{E\in \partial K}\gamma_h \|[\![\bm{u}_h\otimes\bm{n}]\!]\|^{2}_{0,E}
=\nu\sum_{E\in \partial K}\gamma_h \|[\![(\bm{u}_h-\bm{u})\otimes\bm{n}]\!]\|^{2}_{0,E}.
\end{align*}
Summing over all $K\in\mathcal{T}_h$, we have
\begin{align*}
\Big(\sum_{K\in\mathcal{T}_h}\eta_{J_K}^2\Big)^{1/2}
\lesssim \Big(\nu\sum_{E\in \mathcal{E}_h}\gamma_h \|[\![(\bm{u}_h-\bm{u})\otimes\bm{n}]\!]\|^{2}_{0,E}\Big)^{1/2}
\lesssim \vertiii{\bm{u}-\bm{u}_h}.
\end{align*}
\end{proof}

\begin{lemma}\label{efficie12}
Let $(\bm{u},p,\lambda)\in \bm{H}^1_0(\Omega)\times L^2_0(\Omega) \times \mathbb{R}_+$ solve \eqref{stpre11}, and  $(\bm{u}_h,p_h,\lambda_h)\in \bm{V}_h\times Q_h\times\mathbb{R}_+$.
Then we have 
\begin{align*}
\Big(\sum_{K\in\mathcal{T}_h}\eta_{R_K}^2\Big)^{1/2}\lesssim \vertiii{\bm{u}-\bm{u}_h}+\nu^{-1/2}\|p-p_h\|+h.o.t.,
\end{align*}
where $h.o.t.= \nu^{-1/2}\big(\displaystyle{\sum_{K\in\mathcal{T}_h}} h^2_K\|\lambda \bm{u}-\lambda_h \bm{u}_h\|_{0,K}^2\big)^{1/2}$.
\end{lemma}
\begin{proof}
Define the functions $R$ and $W$ locally for any $K\in\mathcal{T}_h$ by
\[
R|_K = \lambda_h\bm{u}_h+\nu\Delta \bm{u}_h-\nabla p_h
\quad\text{and}\quad
W=\nu^{-1}h^{2}_K R b_K
\] 
From \eqref{efficie11} we have
\begin{align*}
\eta^2_{R_K}&=\nu^{-1}h^2_K \|R\|^2_{0,K}\lesssim \int_K R\cdot(\nu^{-1}h^2_K  R b_K) d\bm{x}
=\int_K(\lambda_h\bm{u}_h+\nu\Delta \bm{u}_h-\nabla p_h)\cdot W\,d\bm{x}.
\end{align*}
Note that $\lambda\bm{u}+\nu\Delta \bm{u}-\nabla p=0$. Subtracting this from the last term, using integration by parts
and $W|_{\partial K}=0$, we obtain
\begin{align*}
\eta^2_{R_K}&\lesssim \nu\int_K\nabla( \bm{u}-\bm{u}_h)\cdot\nabla W\,d\bm{x}
+\int_K(p_h-p)\nabla \cdot W\,d\bm{x}+\int_K(\lambda_h \bm{u}_h-\lambda \bm{u})\cdot W\,d\bm{x}.
\end{align*}
Applying Cauchy-Schwarz inequality, implies
\begin{align}
\begin{split}\label{eq:eff1}
\eta^2_{R_K}\lesssim &\left(\nu^{1/2}\|\nabla(\bm{u}-\bm{u}_h)\|_{0,K}+\nu^{-1/2}\|p-p_h\|_{0,K}+ \nu^{-1/2}h_K\|\lambda_h \bm{u}_h-\lambda \bm{u}\|_{0,K}\right)\\
&\quad\left(\nu^{1/2}\|\nabla W\|_{0,K}+\nu^{1/2} h^{-1}_K\|W\|_{0,K}\right).
\end{split}
\end{align}
From \eqref{efficie11} we get
\begin{align*}
\nu^{1/2}\|\nabla W\|_{0,K}+\nu^{1/2} h^{-1}_K\|W\|_{0,K}
\lesssim \nu^{-1/2}h_K\|R\|_{0,K} = \eta_{R_K}.
\end{align*}
Hence, dividing \eqref{eq:eff1} by $\eta_{R_K}$  and taking the square-root of the sum of the squares over all $K\in\mathcal{T}_h$ ends the proof. 
\end{proof}
\par
Let $E$ be an interior edge which is shared by two elements $K_1$ and $K_2$. 
Let $b_E$ denote the standard polynomial edge bubble function for $E$ with support in
$\omega_E=\{K_1,K_2\}$. In case of a regular edge $E$, we choose $\widetilde{K}=K_2$. 
When one vertex of $E$ is a hanging node, then we choose $K_1$ such that $E$ is an entire edge of $K_1$ and define $\widetilde{K}\subset K_2$ as the largest rectangle contained in $K_2$ such that $E$ is one of the entire edges of $\widetilde{K}$. 
We then set $\widetilde{\omega}_E=\{K,\widetilde{K}\}$.
\par
If  $\bm{\sigma}$ is a vector-valued polynomial function on $E$, then 
 \begin{align}\label{efficie131}
 \|\bm{\sigma}\|_{0,E}\lesssim \|b_E^{1/2}\bm{\sigma}\|_{0,E}.
\end{align}   
Moreover we can define an extension $\bm{\sigma}_b\in \bm{H}^{1}_0(\widetilde{\omega}_E)$ such that $\bm{\sigma}_b|_E=b_E\bm{\sigma}$ and from \cite{RV,AMOJT,GKDS}  we have
\begin{align}\label{efficie13}
\begin{split}
\|\bm{\sigma}_b\|_{0,K}&\lesssim h_E^{1/2}\|\bm{\sigma}\|_{0,E} \quad\forall K\in\widetilde{\omega}_{E},\\
\|\nabla\bm{\sigma}_b\|_{0,K}&\lesssim h_E^{-1/2}\|\bm{\sigma}\|_{0,E}\quad \forall K\in\widetilde{\omega}_{E}.
\end{split}
\end{align}
\begin{lemma}\label{lemma:4.9}
Let $(\bm{u},p,\lambda)\in \bm{H}^1_0(\Omega)\times L^2_0(\Omega) \times \mathbb{R}_+$ solve \eqref{stpre11}, and  $(\bm{u}_h,p_h,\lambda_h)\in \bm{V}_h\times Q_h\times\mathbb{R}_+$.
Then we have 
\begin{align*}
\Big(\sum_{K\in\mathcal{T}_h}\eta_{E_K}^2\Big)^{1/2}\lesssim \vertiii{\bm{u}-\bm{u}_h}+\nu^{-1/2}\|p-p_h\|+h.o.t.,
\end{align*}
where $h.o.t.= \nu^{-1/2}\big(\displaystyle{\sum_{K\in\mathcal{T}_h}} h_K^2\|\lambda \bm{u}-\lambda_h \bm{u}_h\|_{0,K}^2\big)^{1/2}$.
\end{lemma}
\begin{proof}
Let for any interior edge $E\in\mathcal{E}_h^i$ the functions $R$ and $\Lambda$ be such that 
\begin{align*}
R|_{E} =  [\![(p_h\bm{I}-\nu\nabla \bm{u}_h)\bm{n}]\!]
\qquad\text{and}\qquad
\Lambda=\nu^{-1}h_ERb_E.
\end{align*}
Using \eqref{efficie131} and $[\![(p\bm{I}-\nabla \bm{u})\bm{n}]\!]|_E=0$  
we get
\begin{align*}
\nu^{-1}h_E \| R\|^{2}_{0,E}\lesssim\int_E R\cdot(\nu^{-1}h_ERb_E)d\bm{s}
=\int_E[\![((p_h-p)\bm{I}-\nabla (\bm{u}_h-\bm{u}))\bm{n}]\!]\cdot\Lambda\,d\bm{s}
\end{align*}
Using Green's formula over each of the two element of $\tilde\omega_E$, gives
\begin{align*}
\int_E[\![((p_h-p)\bm{I}-\nabla (\bm{u}_h-\bm{u}))\bm{n}]\!]\cdot\Lambda\,d\bm{s}
&=\sum_{K\in\widetilde{\omega}_E} \int_K( -\nu \Delta(\bm{u}-\bm{u}_h)+\nabla (p-p_h))\cdot\Lambda\,d\bm{x}\\
&-\sum_{K\in\widetilde{\omega}_E}\int_K(\nu\nabla (\bm{u}-\bm{u}_h)- (p-p_h)\bm{I})\!:\!\nabla\Lambda\,d\bm{x}.
\end{align*}
Using $\lambda\bm{u}+\nu\Delta \bm{u}-\nabla p=0$, we obtain
\begin{align}\label{effjumps}
\begin{split}
\nu^{-1}h_E \| R\|^{2}_{0,E}
&\lesssim \sum_{K\in\widetilde{\omega}_E}\int_K(\lambda_h\bm{u}_h+\nu\Delta \bm{u}_h-\nabla p_h)\cdot \Lambda\,d\bm{x}\\
&\quad+\sum_{K\in\widetilde{\omega}_E}\int_K (\lambda \bm{u}-\lambda_h \bm{u}_h)\cdot \Lambda\,d\bm{x}\\
&\quad+\sum_{K\in\widetilde{\omega}_E} \int_K(-\nu \nabla(\bm{u}-\bm{u}_h)+ (p-p_h)\bm{I})\!:\! \nabla\Lambda\,d\bm{x}\\
&= T_1+T_2+T_3.
\end{split}
\end{align} 
Using Cauchy-Schwarz inequality, shape-regularity of the mesh, and \eqref{efficie13} yields
\begin{align*}
T_1\lesssim \left(\sum_{K\in\widetilde{\omega}_E}\eta^2_{R_K}\right)^{1/2}\left(\sum_{K\in\widetilde{\omega}_E}
\nu  h_{K}^{-2}\|\Lambda\|^{2}_{0,K}\right)^{1/2}
\lesssim \left(\sum_{K\in\widetilde{\omega}_E}\eta^2_{R_K}\right)^{1/2}\!\nu^{-1/2}h_E^{1/2} \| R\|_{0,E},
\end{align*}  
\begin{align*}
T_2\lesssim \left(\sum_{K\in\widetilde{\omega}_E}\left(
\nu^{-1}h_K^2\|\lambda\bm{u}-\lambda_h\bm{u}_h\|^2_{0,K}\right)\right)^{1/2}\nu^{-1/2}h_E^{1/2} \| R\|_{0,E},
\end{align*}
as well as
\begin{align*}
T_3\lesssim \left(\sum_{K\in\widetilde{\omega}_E}\left(\nu\|\nabla(\bm{u}-\bm{u}_h)\|^2_{0,K}
+\nu^{-1}\|p-p_h\|^2_{0,K}\right)
\right)^{1/2}\nu^{-1/2}h_E^{1/2} \| R\|_{0,E}.
\end{align*}
Combining the above estimates $T_1, T_2$ and $T_3$, dividing \eqref{effjumps} by $\nu^{-1/2}h_E^{1/2} \| R\|_{0,E}$  and
summing over all interior edges of all $K\in\mathcal{T}_h$ the desired result is proven by the finite overlap of the patches
$\widetilde{\omega}_E$ and Lemma~\ref{efficie12}.
\end{proof}
 \begin{theorem}\label{efficie}
 Let $(\bm{u},p,\lambda)\in \bm{H}^1_0(\Omega)\times L^2_0(\Omega) \times \mathbb{R}_+$ be the solution of the Stokes eigenvalue problem \eqref{stpre11}  and $(\bm{u}_h,p_h,\lambda_h)\in\bm{V}_h\times Q_h\times\mathbb{R}_+$ the $H^{\textrm{div}}$-DG approximation obtained by \eqref{hdiv11}. Then the a posteriori error estimator $\eta_h$ is efficient in the sense that
 \begin{align}
\eta_h \lesssim \vertiii{\bm{u}-\bm{u}_h}+\nu^{-1/2}\|p-p_h\|_{0} + h.o.t.,
\end{align}   
where $h.o.t. =\nu^{-1/2}\big(\displaystyle{\sum_{K\in\mathcal{T}_h}}h^2_K\|\lambda\bm{u}-\lambda_h\bm{u}_h\|_{0,K}^2\big)^{1/2}$.
\end{theorem}
\begin{proof}
The statement follows from a combination of Lemma~\ref{efficiej11}--\ref{lemma:4.9}.
\end{proof}
 \begin{corollary}\label{cor12}
If $\bm{u}\in \bm{H}^2(\mathcal{T}_h)$ and $p\in H^{1}(\mathcal{T}_h)$,  then the eigenvalue error satisfies
 \begin{align*}
 \eta_h^2\lesssim |\lambda-\lambda_h|+
  \lambda||\bm{u}-\bm{u}_h||_0^2+(h.o.t.)^2,
 \end{align*}
 where $h.o.t.=\nu^{-1/2}\big(\displaystyle{\sum_{K\in\mathcal{T}_h}}h^2_K\|\lambda\bm{u}-\lambda_h\bm{u}_h\|_{0,K}^2\big)^{1/2}$.
\end{corollary}
\begin{proof}
Combining Theorems \ref{riziden11} \& \ref{efficie} implies 
 \begin{align*}
 \eta_h^2\lesssim |\lambda-\lambda_h|+ \nu^{-1}\|p-p_h\|_{0}^2+|C_h(\bm{u}-\bm{u}_h,\bm{u}-\bm{u}_h)|+\lambda||\bm{u}-\bm{u}_h||_0^2+(h.o.t.)^2.
 \end{align*}
 Let $C_{\#}$ be defined as
 \begin{align}
 C_{\#} =\frac{\nu^{-1}\|p-p_h\|_{0}^2+|C_h(\bm{u}-\bm{u}_h,\bm{u}-\bm{u}_h)|}{|\lambda-\lambda_h|}.\nonumber
 \end{align}
 From the estimates in \cite[Section 8]{DSCSAT} and  the eigenvalue estimate (\ref{eigenrr11}), we can conclude that $\nu^{-1}\|p-p_h\|_{0}^2+|C_h(\bm{u}-\bm{u}_h, \bm{u}-\bm{u}_h)|$
is of the same order as $|\lambda-\lambda_h|$. Hence $C_\#$ can be bounded from above by a uniform constant.
Then it holds that
 \begin{align}
  \eta_h^2\lesssim (1+C_\#)|\lambda-\lambda_h|+\lambda||\bm{u}-\bm{u}_h||_0^2+(h.o.t.)^2.
 \end{align}
This completes the proof.
\end{proof}

\section{Numerical experiments}\label{comre}
This section is devoted to several numerical experiments on one convex and two non-convex domains.
The experiments verify reliability and efficiency of the proposed a posteriori error estimator of Section \ref{aposterrestana} for the eigenvalue
error of the smallest (simple) eigenvalue and up to polynomial degree 3.
\par
We employ the standard adaptive finite element loop with the steps \emph{solve}, \emph{estimate}, \emph{mark} and \emph{refine}.
To solve the algebraic eigenvalue problem we use the ARPACK library \cite{ARPACK} in combination with a direct solver.
We mark elements of the mesh for refinement on the level $\ell$ in a minimal set $\mathcal{M}_\ell$ using the bulk marking strategy \cite{Doerfler} with bulk parameter $\theta=1/2$,
i.e.  $\mathcal{M}_\ell$ is the minimal set such that $\theta\sum_{K\in\mathcal{T}_\ell}\eta_{K}^2 \leq \sum_{K\in\mathcal{M}_\ell}\eta_{K}^2$. The mesh is refined with one level irregular nodes.
The implementation of the method is done in the software library amandus \cite{GK}, which is based on 
the dealii finite element library \cite{WTLGMM}.
\par
In all experiments we consider the viscosity $\nu=1$ and
chose the penalty parameter $\gamma=k(k+1)/2$ for $k$-th order $RT_k\times Q_k$ finite element pairs, $k=1,2,3$.
Since the eigenvalues of the Stokes problem are related to the eigenvalues
of the buckling eigenvalue problem of clamped plates via the stream function formulation,
we can use reference values for the eigenvalues from \cite{BPEBPT,BMS2014,SZ2017}.

\subsection{Square domain}\label{testsquared} 
\begin{figure}[tbp]
\centering
\subfigure[]{
\includegraphics[width= 0.4\textwidth]{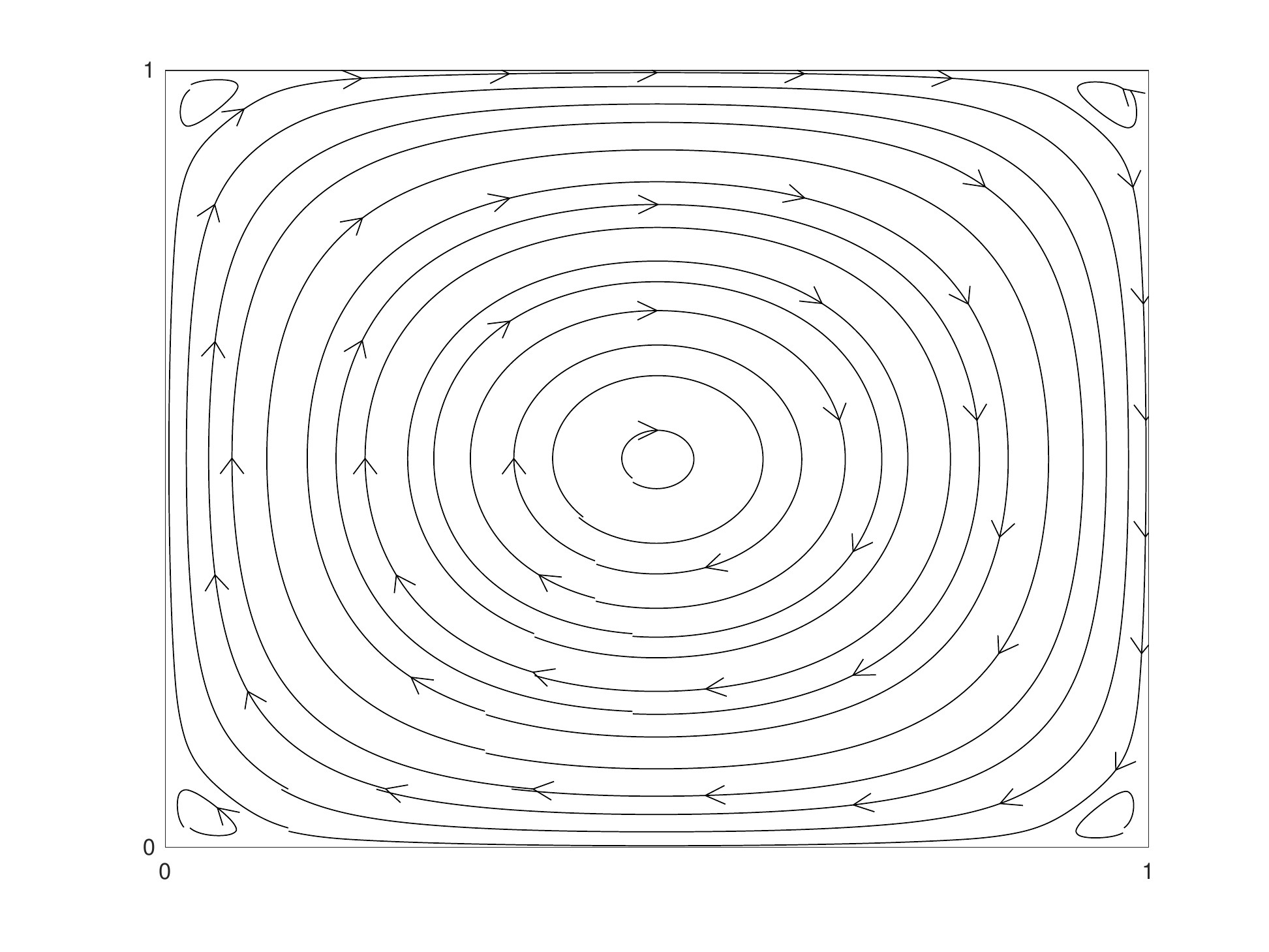}
\label{figS22a}
}
\hspace{0.08\textwidth}
\subfigure[]{
\includegraphics[width= 0.4\textwidth]{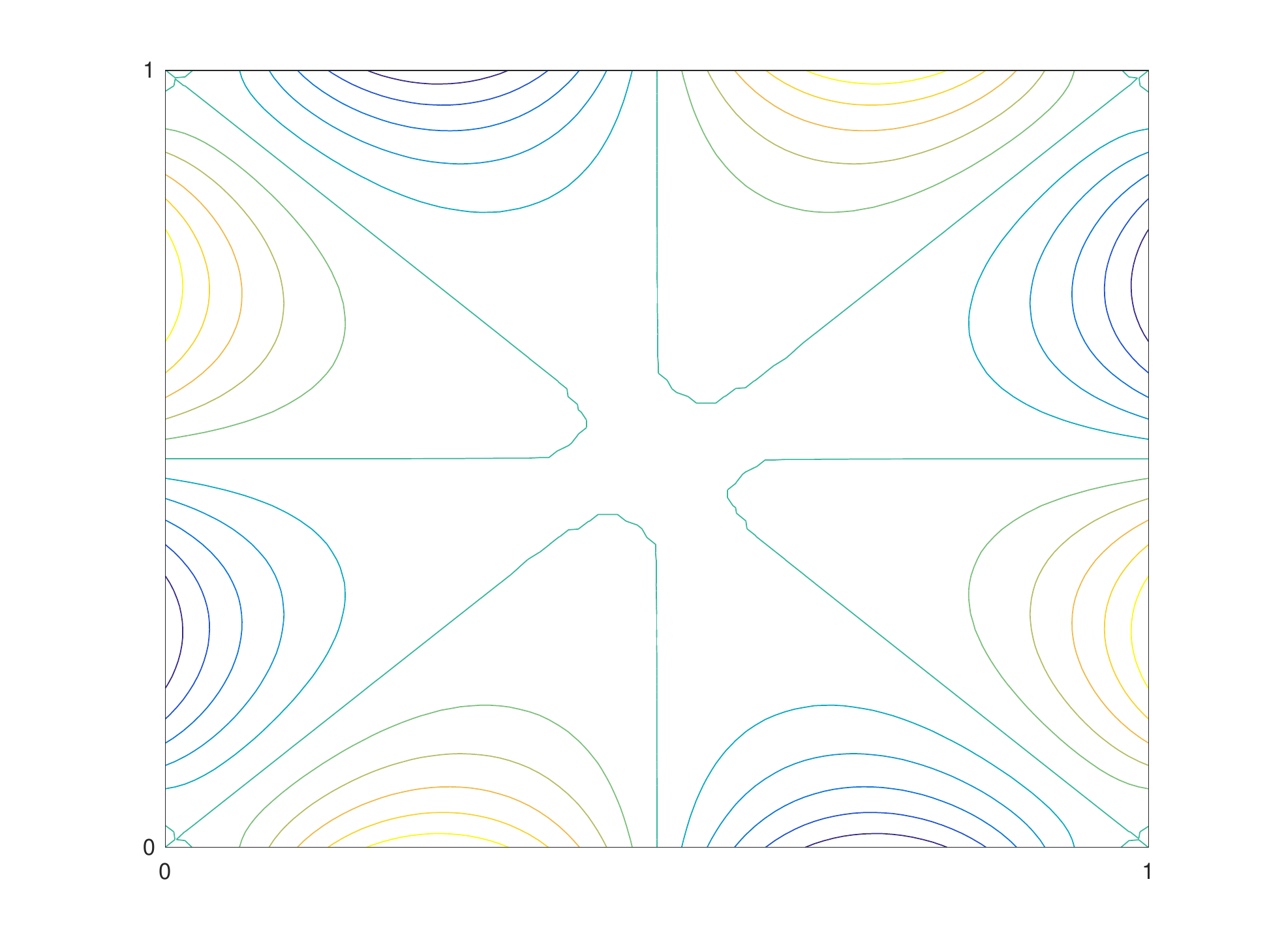}
\label{figS22c}
}
\caption{Streamline plot of the discrete eigenfunction $\bm{u}_\ell$ \subref{figS22a}, and plot of the discrete pressure $p_\ell$  \subref{figS22c}.}
\label{figexS212}
\end{figure}
\begin{figure}[tbp]
\centering
\includegraphics[width=\textwidth]{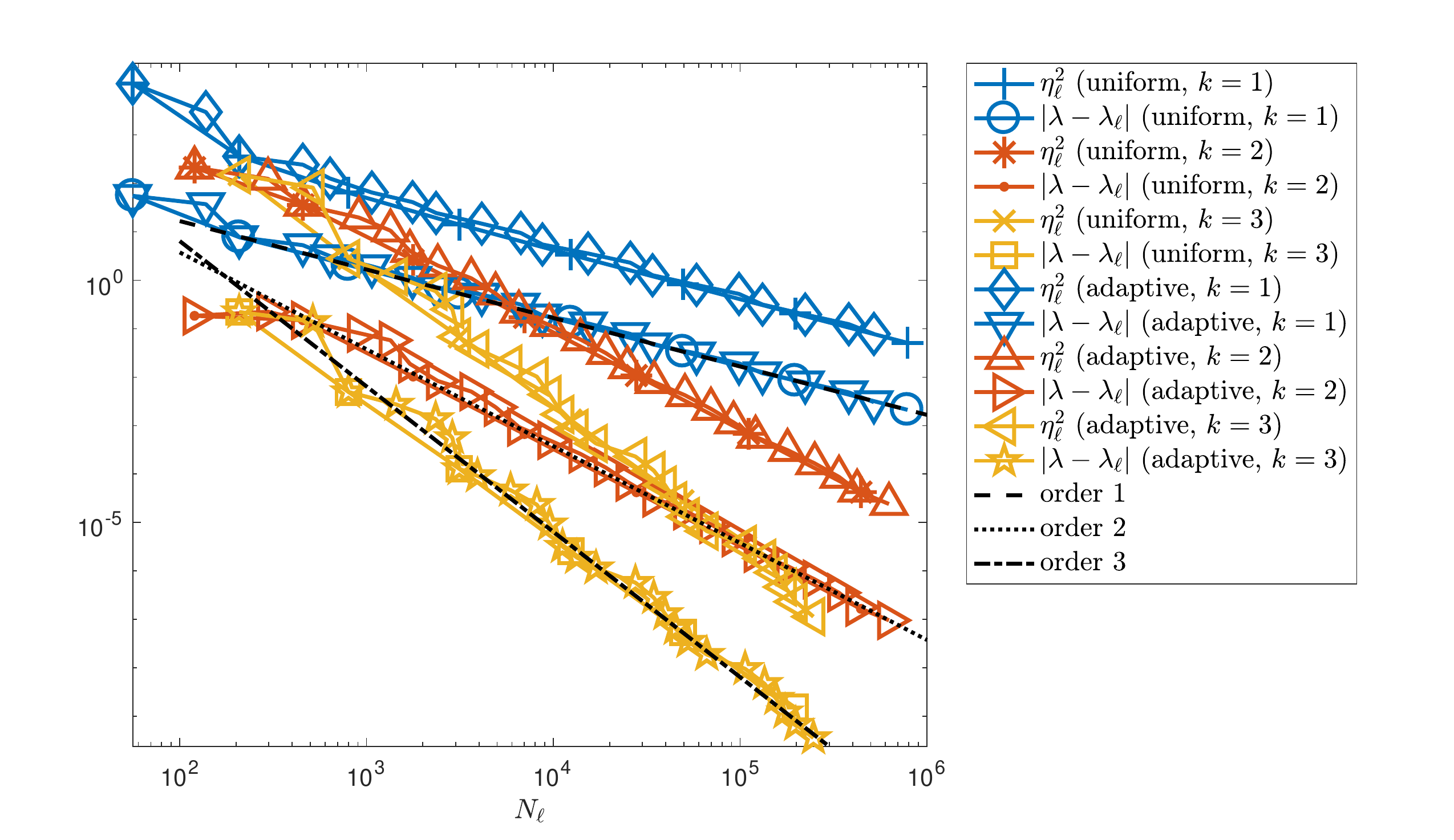}
\caption[]
{Convergence history of $|\lambda-\lambda_\ell|$ and $\eta_\ell^2$ on uniformly and adaptively refined meshes for the square domain.}
\label{figex21}
\end{figure}
In this example, we consider the square domain $\Omega=(0,1)^2$.  
The reference value for the first eigenvalue reads $\lambda=52.344691168$ \cite{BPEBPT,BMS2014,SZ2017}.
The streamline plot of the discrete eigenfunction $\bm{u}_\ell$ and the plot of the discrete pressure $p_\ell$ on a uniform mesh for $k=1$ are displayed in Figures~\ref{figS22a} 
and \ref{figS22c}, respectively.
In Figure~\ref{figex21}, we observe that both uniform and adaptive mesh refinement leads to optimal
orders of convergence $\mathcal{O}(N_\ell^{-k})$ for the eigenvalue error $|\lambda-\lambda_\ell|$.
This is due to the fact that the domain is convex and the first eigenfunction is smooth enough.
Note that for uniform meshes $\mathcal{O}(N_\ell^{-k})\approx \mathcal{O}(h^{2k})$, for $N_\ell = \text{dim}(V_h\times Q_h)$.
We observe that the convergence graphs for uniform and adaptive mesh refinement overlap each other for
both the eigenvalue errors $|\lambda-\lambda_\ell|$ as well the a posteriori error estimators $\eta_\ell^2$.
Moreover, we confirm that the a posteriori error estimator $\eta_\ell^2$ is numerically reliable and efficient.

\subsection{L-shaped domain}\label{testlsha}
\begin{figure}[tbp]
\centering
\subfigure[]{
\includegraphics[width=0.4\textwidth]{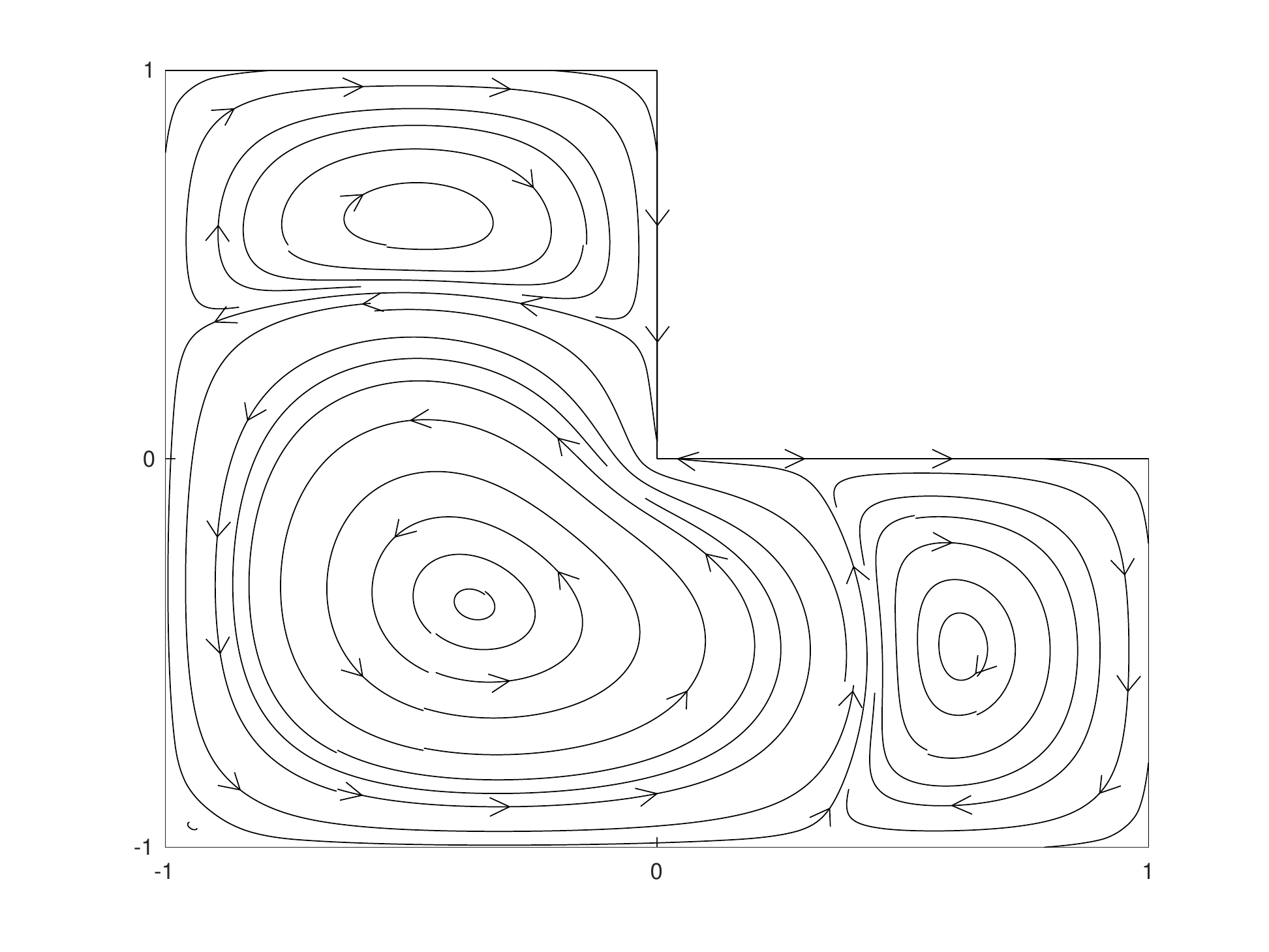}
\label{figL23a1}
}
\hspace{0.08\textwidth}
\subfigure[]{
\includegraphics[width=0.4\textwidth]{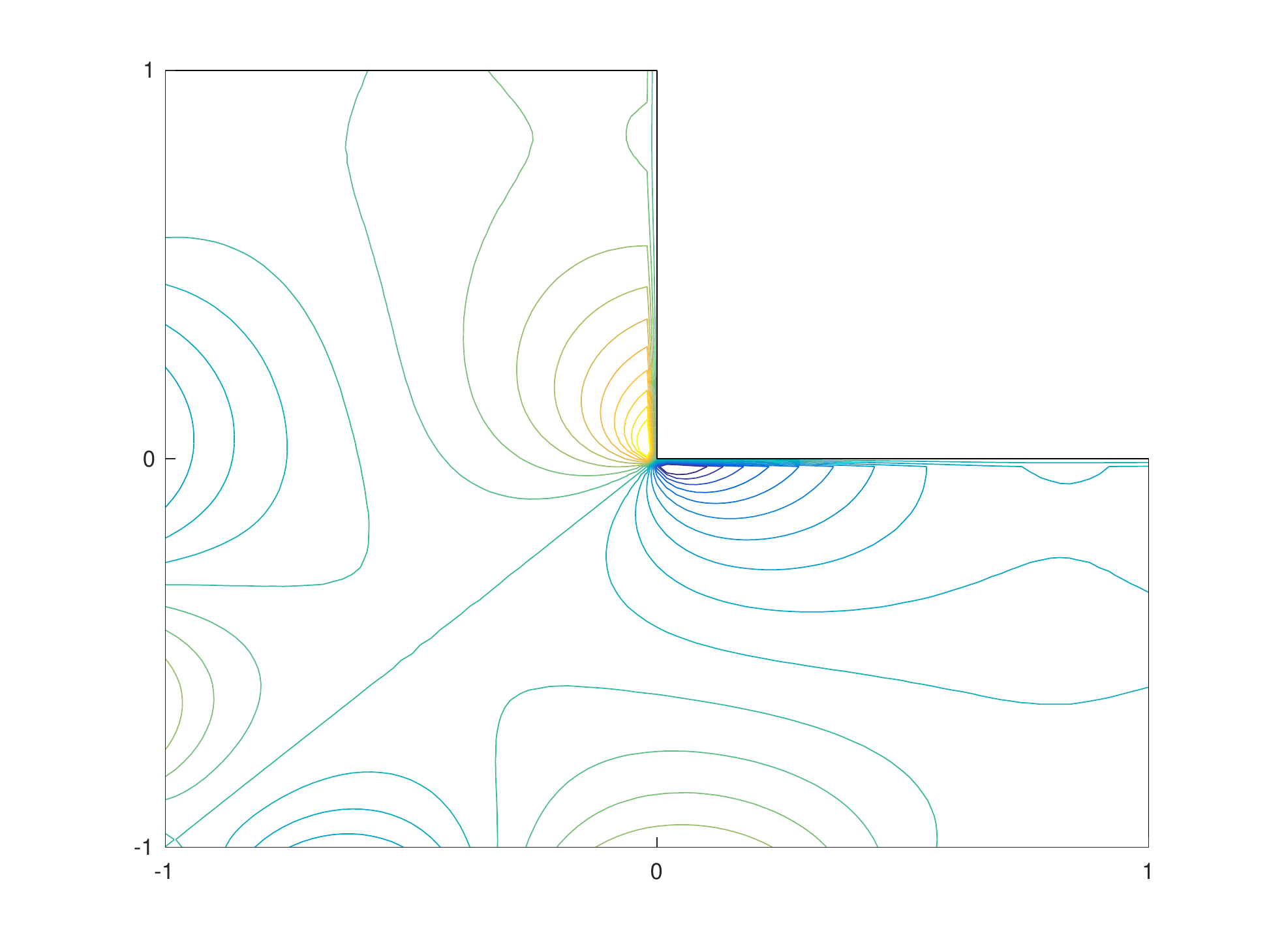}
\label{figL23c1}
}
\caption{Streamline plot of the discrete eigenfunction $\bm{u}_\ell$ \subref{figL23a1} , and plot of the of discrete pressure $p_\ell$  \subref{figL23c1}.}
\label{figexL2121}
\end{figure}
\begin{figure}[tbp]
\centering
\includegraphics[width=\textwidth]{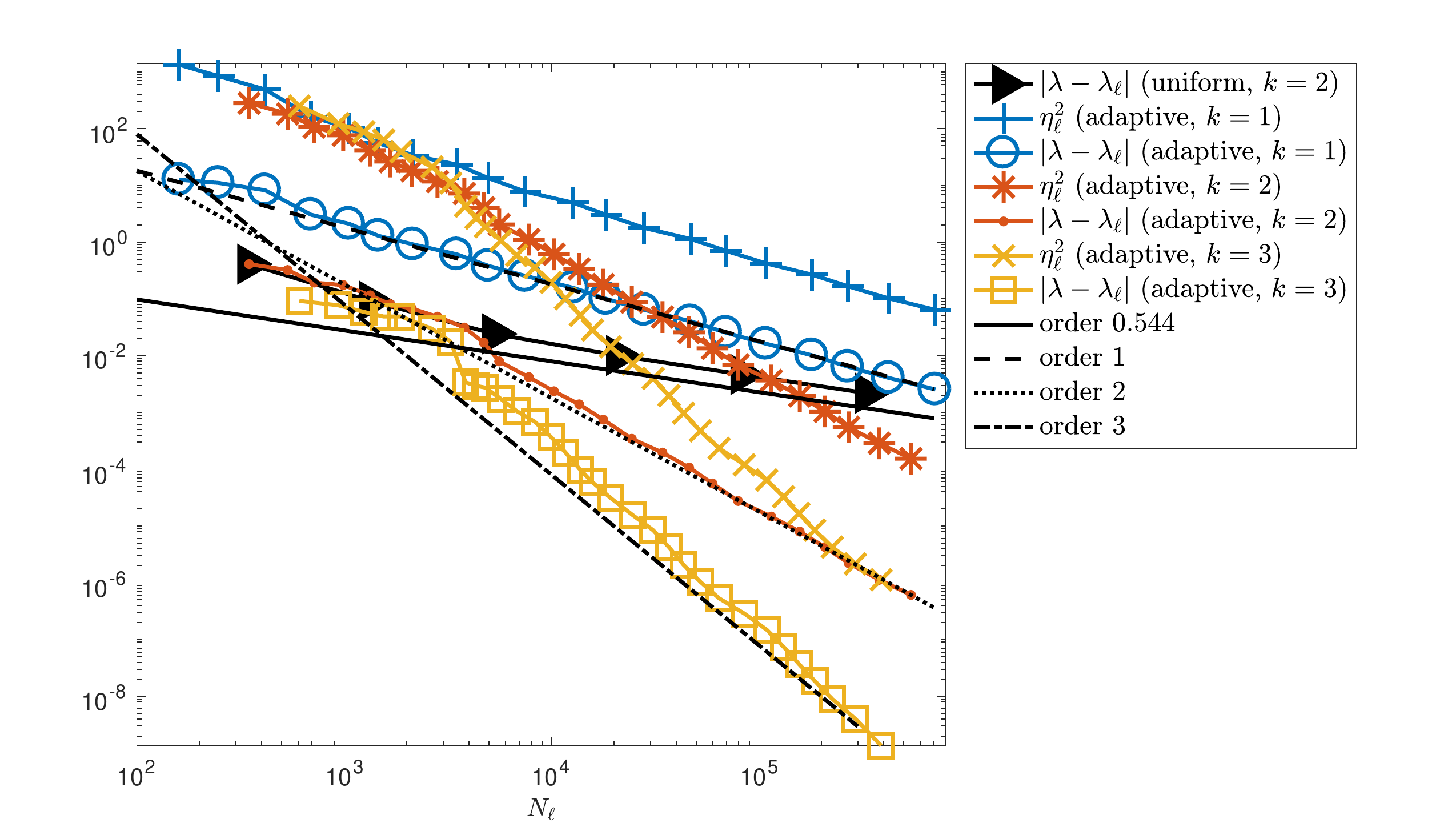}
\caption{Convergence history of $|\lambda-\lambda_\ell|$ and $\eta_\ell^2$ on uniformly and adaptively refined meshes for the L-shaped domain.}
\label{figexL21}
\end{figure}
\begin{figure}[tbp]
\centering
\subfigure[]{
\includegraphics[width=0.3\textwidth]{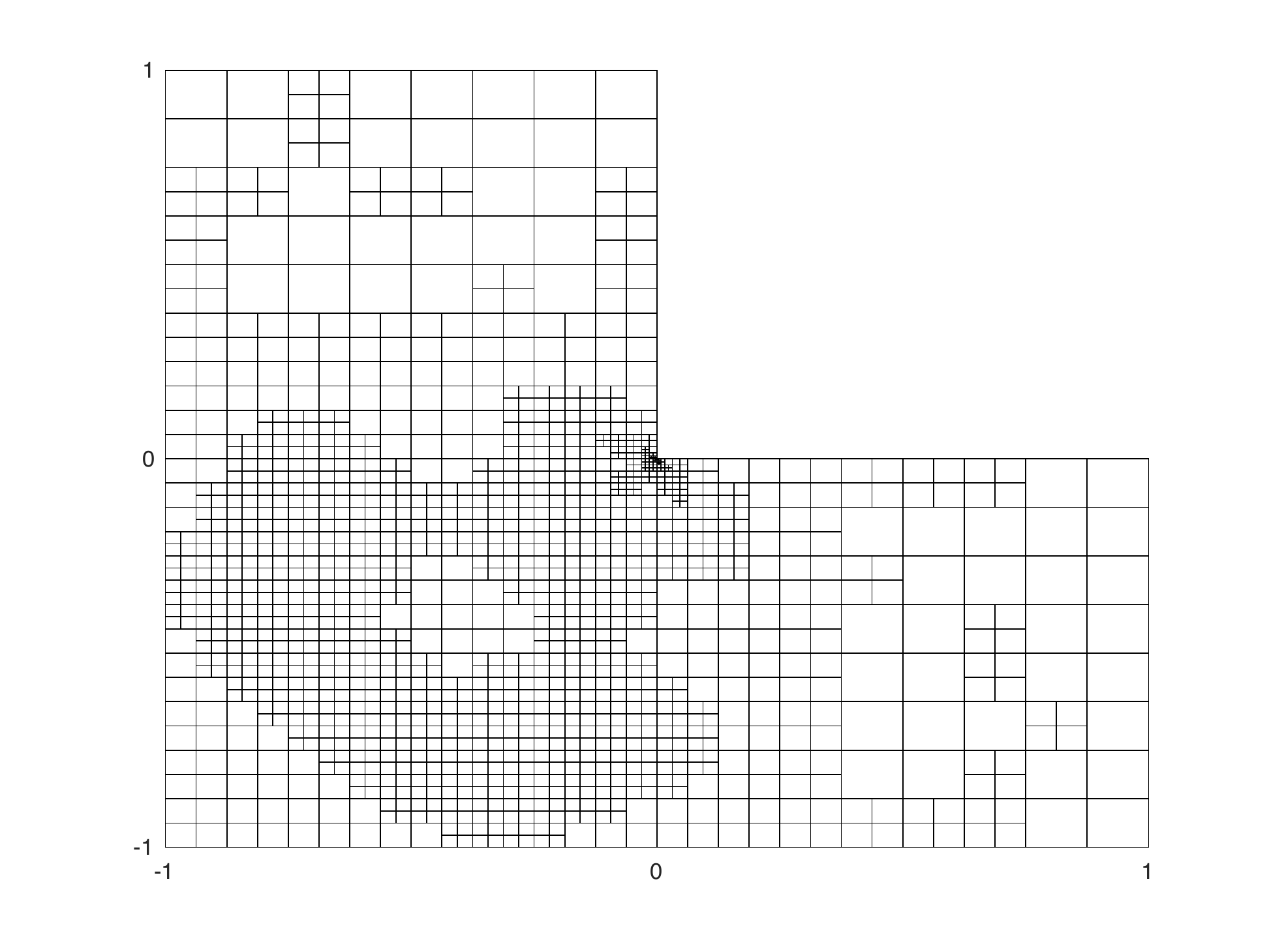}
\label{figL22a}
}
\subfigure[]{
\includegraphics[width=0.3\textwidth]{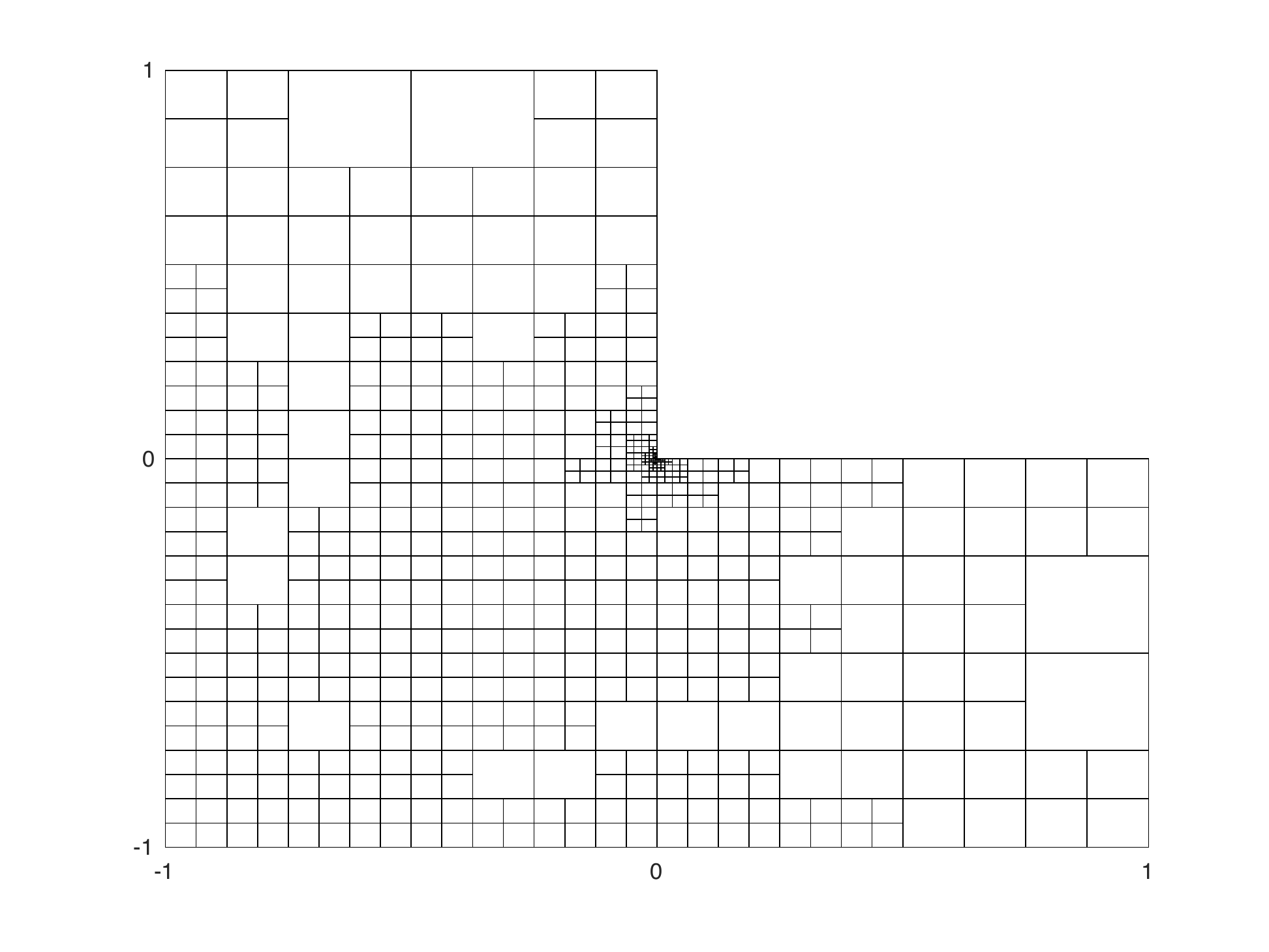}
\label{figL22c}
}
\subfigure[]{
\includegraphics[width=0.3\textwidth]{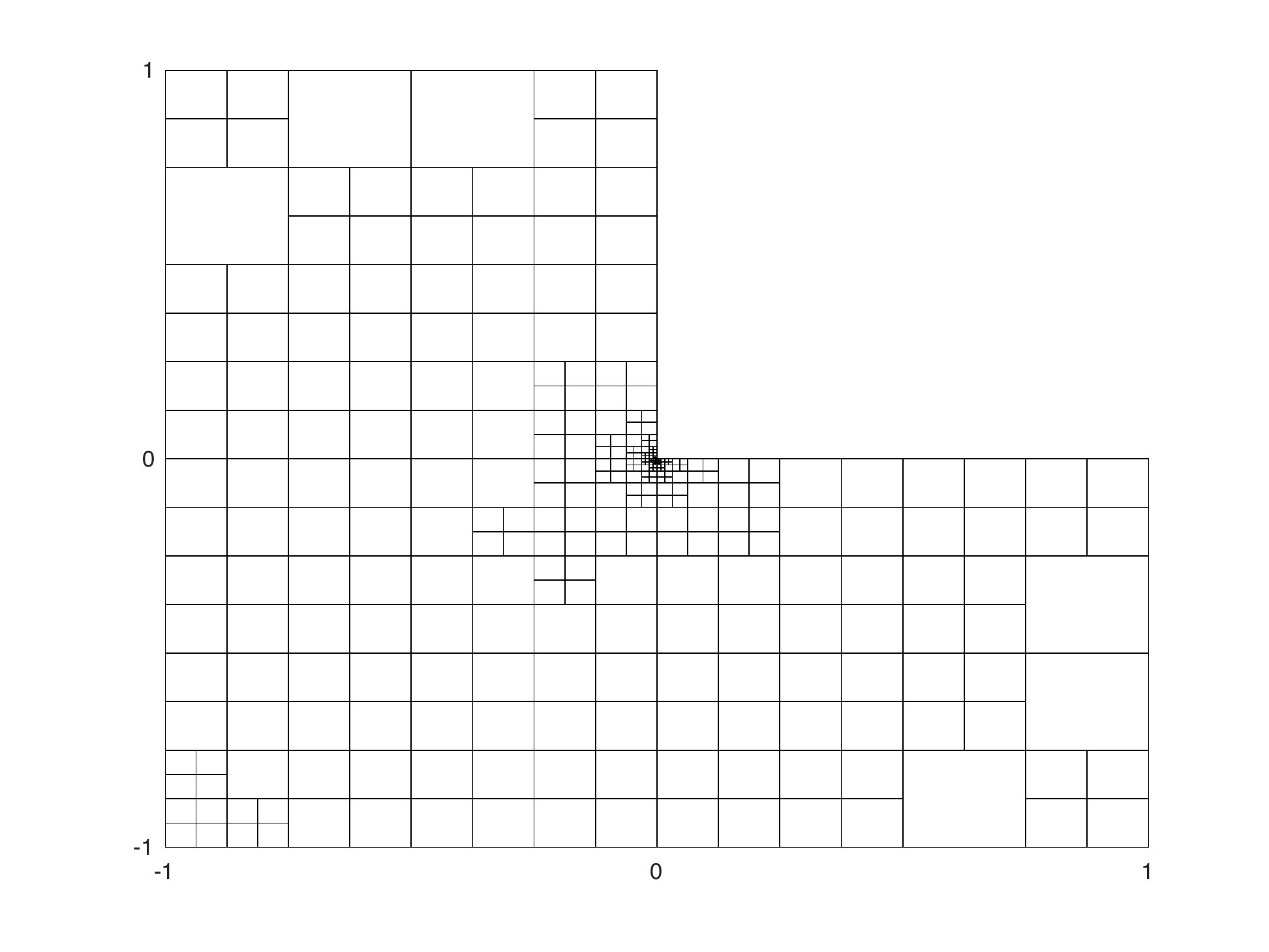}
\label{figL22c1}
}
\caption{Adaptively refined meshes for  $RT_1\times Q_1$ \subref{figL22a}, $RT_2\times Q_2$ \subref{figL22c}, and $RT_3\times Q_3$ \subref{figL22c1}.}
\label{figexL212}
\end{figure}
In the second example, we take the non-convex L-shaped domain $\Omega=(-1,1)^2\setminus (0,1)^2$
with a re-entrant corner at the origin, which allows for singular eigenfunctions.
To compute the error of the first eigenvalue,  we take $\lambda=32.13269465$ as a reference value.
Figures \ref{figL23a1} and \ref{figL23c1} show the computed velocity and discrete pressure as a streamline plot on a uniform mesh computed with $k=1$.
The exponent for the singular function at the re-entrant corner is known to be $\alpha\approx0.544483736782464$.
Hence, in Figure \ref{figexL21}  we observe suboptimal convergence of $\mathcal{O}(N_\ell^{-0.544})$ for the eigenvalue error even for $k=2$.
Adaptive mesh refinement however, achieves optimal convergence $\mathcal{O}(N_\ell^{-k})$ of the eigenvalue error for $k=1,2,3$.
The a posteriori error estimator $\eta_\ell^2$ shows to be reliable and efficient in all experiments.
Observe that the eigenvalue error obtained with $k=3$ on adaptively refined meshes is about 6 orders of magnitude smaller than
that for uniform mesh refinement. This demonstrates the importance of mesh adaptivity, in particular for high order methods.
Figures \ref{figL22a} and \ref{figL22c} show some adaptively refined meshes for $k=1,2,3$, which
show strong refinement towards the origin. 
 
\subsection{Slit domain}\label{testslitd} 
\begin{figure}[tbp]
\centering
\subfigure[]{
\includegraphics[width=0.4\textwidth]{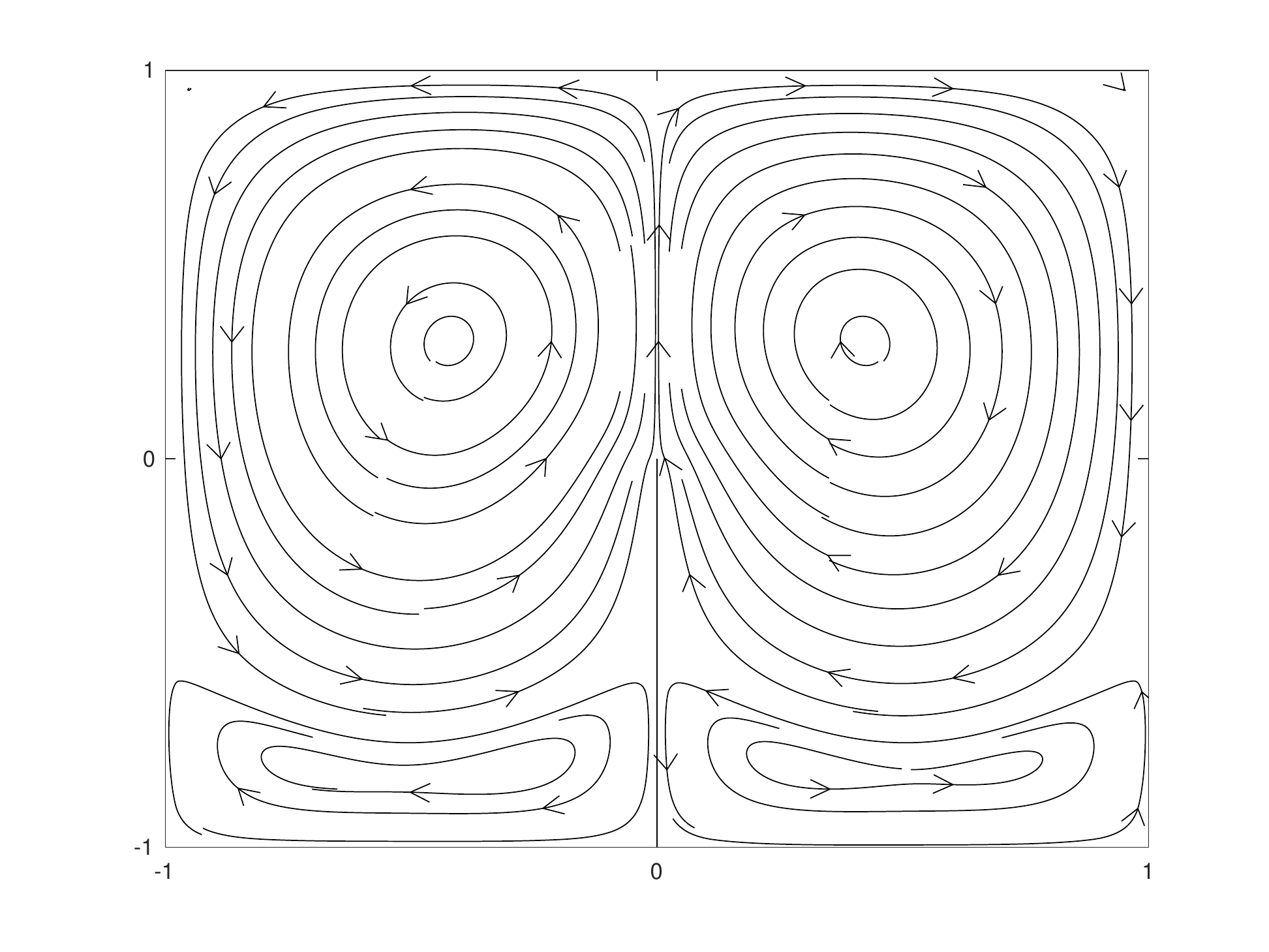}
\label{figSl23a1}
}
\hspace{0.08\textwidth}
\subfigure[]{
\includegraphics[width=0.4\textwidth]{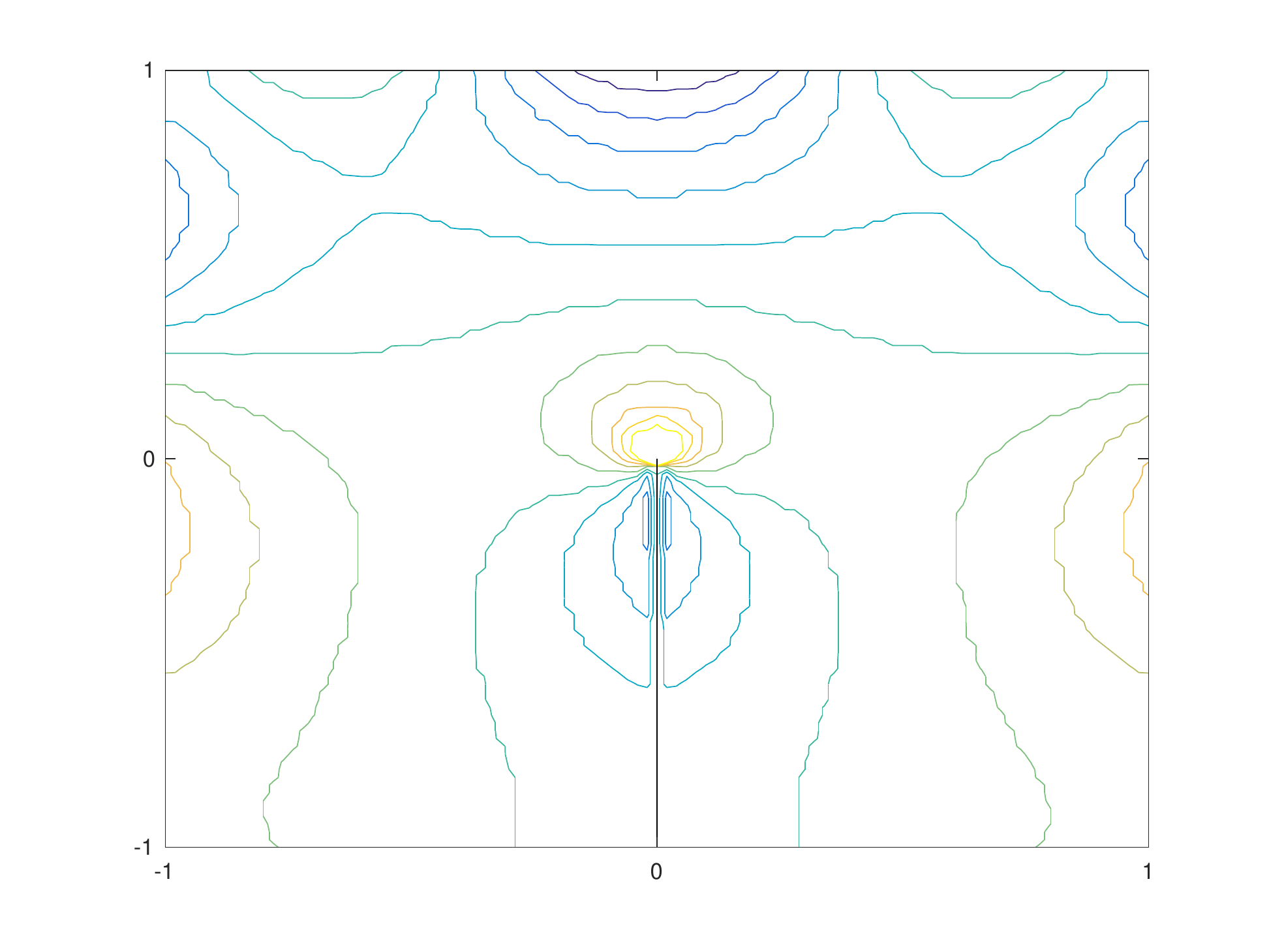}
\label{figSl23c1}
}
\caption{Streamline plot of the discrete eigenfunction $\bm{u}_\ell$ \subref{figSl23a1} , and plot of the discrete pressure $p_\ell$ \subref{figSl23c1}.}
\label{figexSl2121}
\end{figure}
\begin{figure}[tbp]
\centering
\includegraphics[width=\textwidth]{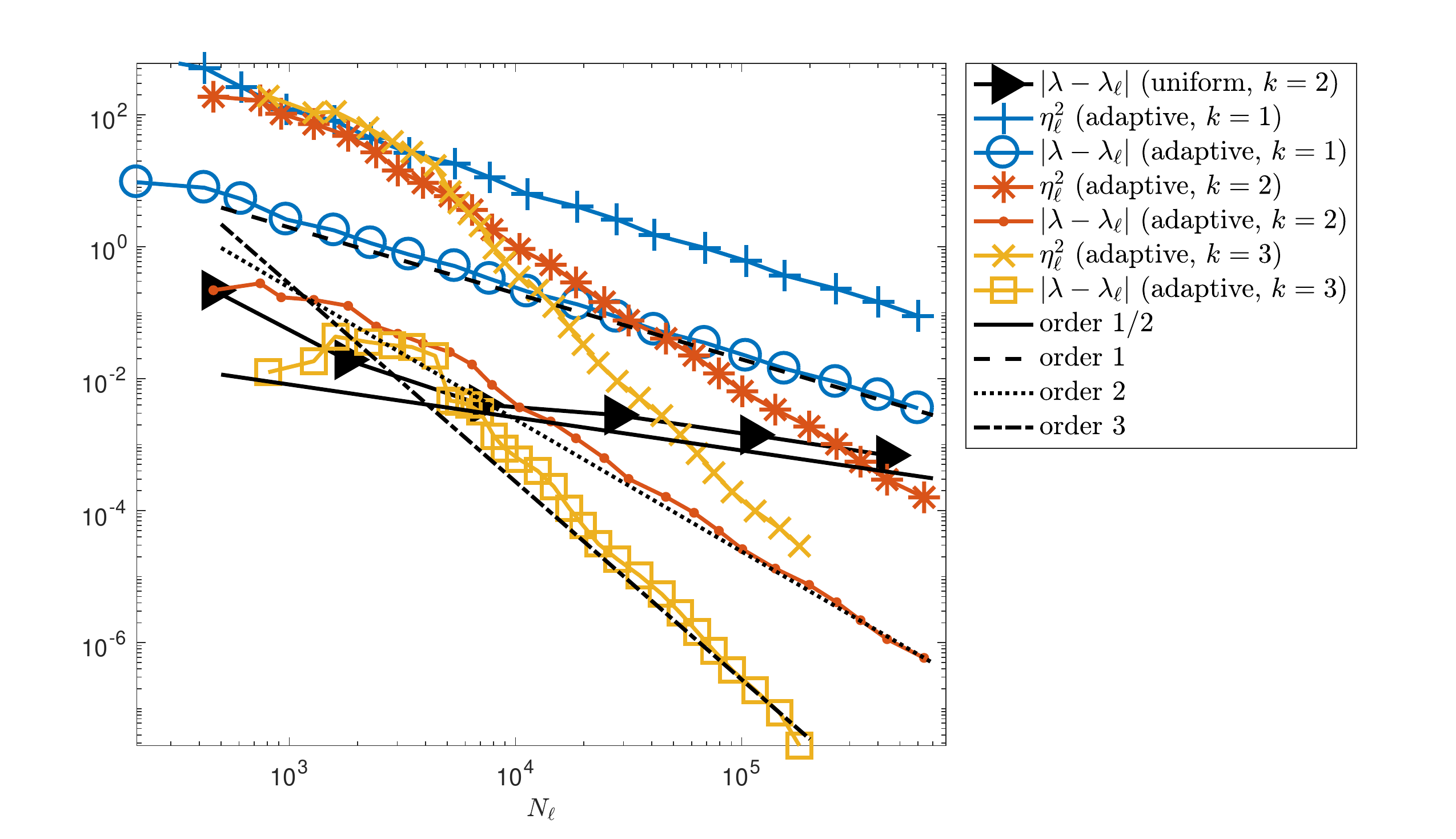}
\caption{Convergence history of $|\lambda-\lambda_\ell|$ and $\eta_\ell^2$ on uniformly and adaptively refined meshes for the slit domain.}
\label{figex2Sl1}
\end{figure}
\begin{figure}[tbp]
\centering
\subfigure[]{
\includegraphics[width=0.3\textwidth]{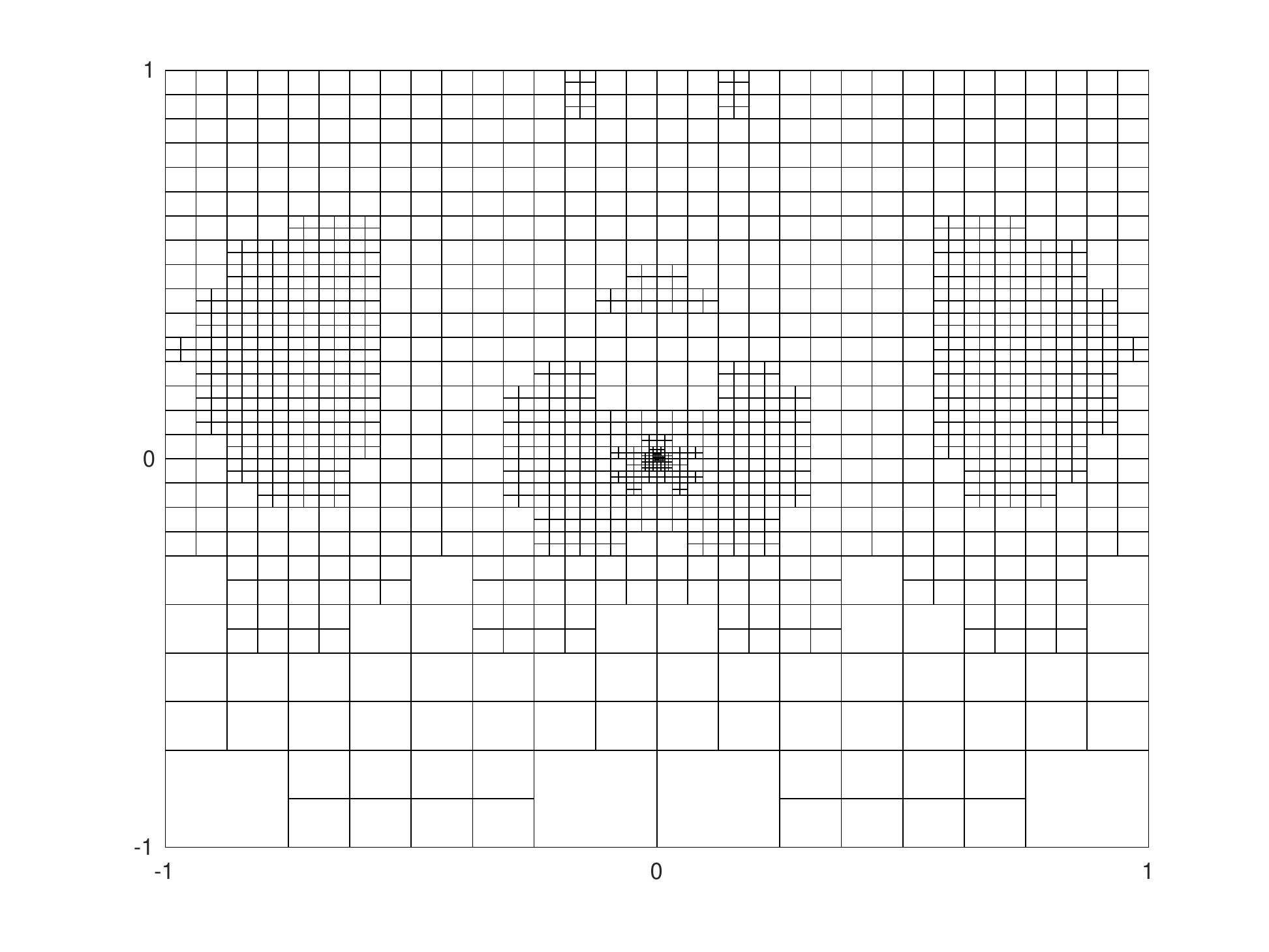}
\label{figSl22a}
}
\subfigure[]{
\includegraphics[width=0.3\textwidth]{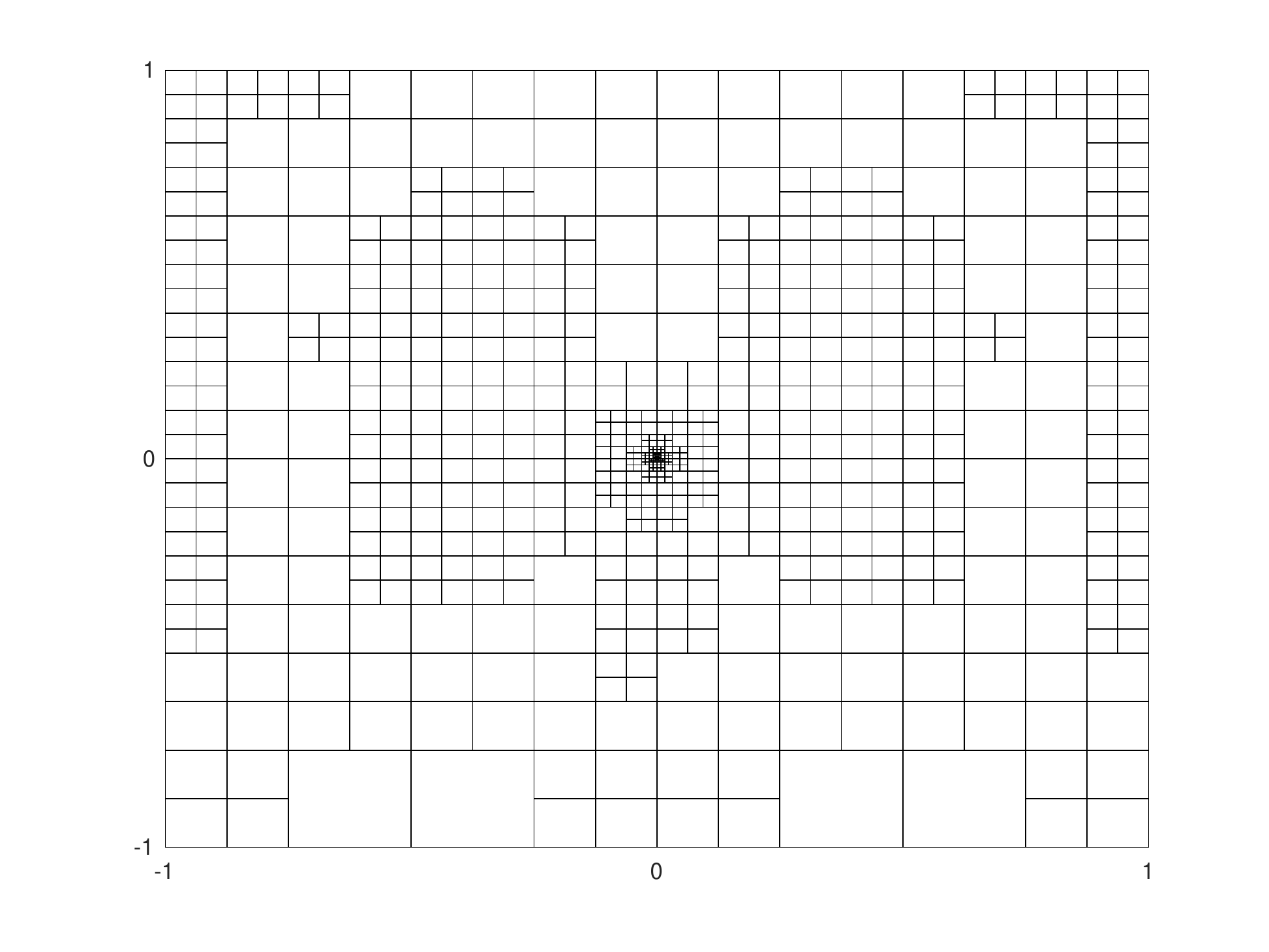}
\label{figSl22c}
}
\subfigure[]{
\includegraphics[width=0.3\textwidth]{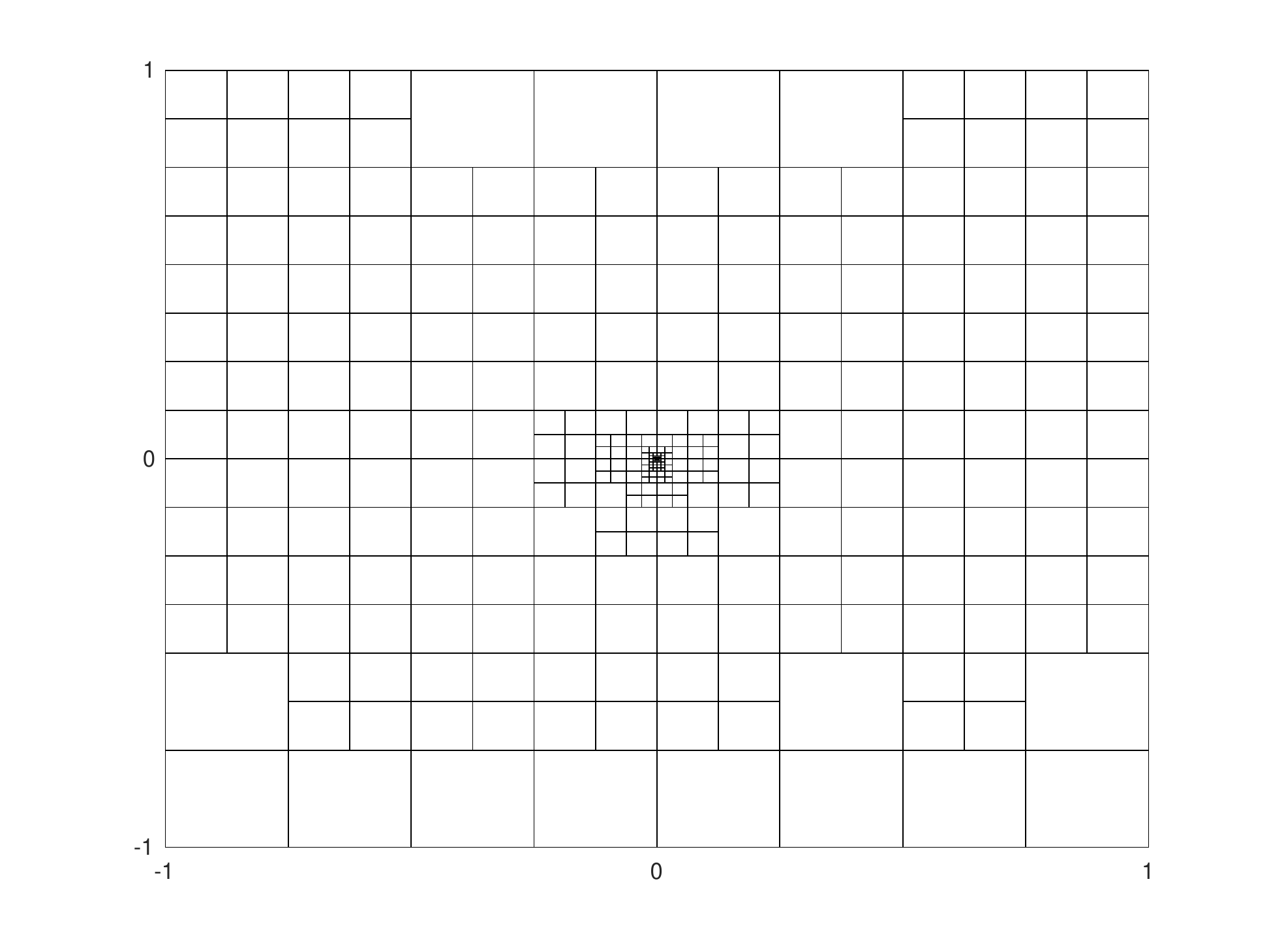}
\label{figSl22c1}
}
\caption{Adaptively refined meshes for $RT_1\times Q_1$ \subref{figSl22a}, $RT_2\times Q_2$ \subref{figSl22c}, and $RT_3\times Q_3$ \subref{figSl22c1}.}
\label{figexSl212}
\end{figure}
In the last example,  let $\Omega=(-1,1)^2\setminus ( \{0\}\times(-1,0) )$ be the slit domain.
To compute the error of the first eigenvalue,  we take $\lambda=29.9168629$ as a reference value. 
The discrete velocity eigenfunction $\bm{u}_\ell$ and discrete pressure $p_\ell$ are displayed in Figures \ref{figSl23a1} and \ref{figSl23c1} as a streamline plot on a uniform mesh for $k=1$. 
In Figure~\ref{figex2Sl1}, we observe suboptimal convergence of $\mathcal{O}(N_\ell^{-1/2})$ for the eigenvalue error on uniform meshes,
but optimal convergence of $\mathcal{O}(N_\ell^{-k})$ for $k=1,2,3$, for adaptively refined meshes.
Moreover, the a posteriori error estimator proves to be numerically reliable and efficient.
Note that we had to stop the third order method on adaptively refined meshes earlier than for the lower order methods, since the accuracy of the reference value has been already reached
with less than $2\cdot 10^5$ degrees of freedom.
Figures \ref{figSl22a} and \ref{figSl22c} show some adaptively refined meshes for $k=1,2,3$,
which are strongly refined towards the tip of the slit at the origin.


\end{document}